\documentclass[reqno,12pt]{amsart}

\usepackage[top=2.0cm,bottom=2.0cm,left=3cm,right=3cm]{geometry}
\usepackage{amsthm,amsmath,amssymb,dsfont}
\usepackage{mathrsfs,amsfonts,functan,extarrows,mathtools}
\usepackage[colorlinks]{hyperref}

\usepackage{marginnote}
\usepackage{xcolor}
\usepackage{stmaryrd}
\usepackage{esint}
\usepackage{graphicx}

\newtheorem{theorem}{Theorem}[section]
\newtheorem{proposition}{Proposition}[section]

\newtheorem{remark}{Remark}[section]
\newtheorem{lemma}{Lemma}[section]

\numberwithin{equation}{section}
\allowdisplaybreaks

\def\d{\mathrm{d}}

\def\no{\nonumber}
\def\R{\mathbb{R}}
\def\eps{\epsilon}
\def\div{\mathrm{div}}

\def\dr{\mathrm{d}}
\def\v{\mathrm{u}}

\def\l{\left\langle}
\def\r{\right\rangle}

\def\A{\mathrm{A}}
\def\N{\mathrm{N}}
\def\g{\mathrm{g}}
\def\B{\mathrm{B}}

\newcounter{wronumber}

\begin{document}
\title[Parabolic-hyperbolic incompressible liquid crystal model]
			{On well-posedness of Ericksen-Leslie's hyperbolic incompressible liquid crystal model}

\author[N. Jiang]{Ning Jiang}
\address[Ning Jiang]{\newline School of Mathematics and Statistics, Wuhan University, Wuhan, 430072, P. R. China}
\email{njiang@whu.edu.cn}

\author[Y.-L. Luo]{Yi-Long Luo}
\address[Yi-Long Luo]
		{\newline School of Mathematics and Statistics, Wuhan University, Wuhan, 430072, P. R. China}
\email{yl-luo@amss.ac.cn}
\thanks{ \today}

\maketitle

\begin{abstract}
 We study the Ericksen-Leslie's hyperbolic incompressible liquid crystal model. Under some constraints on the Leslie coefficients which ensure the basic energy law is dissipative, we prove the local-in-time existence and uniqueness of the classical solution to the system with finite initial energy. Furthermore, with an additional assumption on the coefficients which provides a damping effect, and the smallness of the initial energy, the unique global classical solution can be established.
\end{abstract}





\section{Introduction}

The hydrodynamic theory of incompressible liquid crystals was established by Ericksen \cite{Ericksen-1961-TSR, Ericksen-1987-RM, Ericksen-1990-ARMA} and Leslie \cite{Leslie-1968-ARMA, Leslie-1979} in the 1960's (see also Section 5.1 of \cite{Lin-Liu-2001} ). The so-called Ericksen-Leslie system consists of the following equations of $(\rho(x,t), \v(x,t), \dr(x,t)$, where $(x,t)\in \mathbb{R}^n\times \mathbb{R}^+$ with $n \geq 2$ :
\begin{align}\label{EL-general}
  \left\{ \begin{array}{c}
    \rho \dot\v = \rho \mathrm{F} + \div \hat{\sigma}\,, \div \v = 0\,,\\
    \rho_1 \dot\omega = \rho_1 \mathrm{G} + \hat{\mathrm{g}} + \div \pi\,.
  \end{array}\right.
\end{align}
The system \eqref{EL-general} represents the conservations laws of linear momentum and angular momentum respectively. Here, $\rho$ is the constant fluid density, $\rho_1 \geq 0$ is the inertial constant,  $\v = (\v_1,\cdots, \v_n)^\top$ is the flow velocity, $\dr = (\dr_1,\cdots, \dr_n)^\top$ is the direction field of the liquid molecules with the constraint $|\dr|=1$. Furthermore, $\hat{\mathrm{g}}$ is the intrinsic force associated with $\dr$, $\pi$ is the director stress, $\mathrm{F}$ and $\mathrm{G}$ are external body force and external director body force, respectively. The superposed dot denotes the material derivative $\partial_t + \v\cdot\nabla$, and
$$\omega = \dot \dr = \partial_t \dr + \v\cdot \nabla \dr$$
represents the material derivative of $\dr$.

The constitutive relations for $\hat{\sigma}$, $\pi$ and $\hat{\g}$ are given by:
\begin{equation}\label{Constitutive}
  \begin{aligned}
    \hat{\sigma}_{ij}& = -p\delta_{ij} + \sigma_{ij} - \rho \tfrac{\partial W}{\partial\dr_{k,i}}\dr_{k,j}\,,\\
    \pi_{ij}& = \beta_i \dr_j + \rho \tfrac{\partial W}{\partial\dr_{j,i}}\,,\\
    \hat{\g}_{ij} & = \gamma \dr_i - \beta_j \dr_{i,j} - \rho \tfrac{\partial W}{\partial \dr_i} + \g_i\,.
  \end{aligned}
\end{equation}
Here $p$ is the pressure, the vector $\beta=(\beta_1\,,\cdots\,,\beta_n)^\top$ and the scalar function $\gamma$ are Lagrangian multipliers for the constraint $|\dr|=1$, and $W$ is the Oseen-Frank energy functional for the equilibrium configuration of a unit director field:
\begin{equation}\label{Oseen-Frank-energy}
  \begin{aligned}
    2W =& k_1 (\div \dr)^2 + k_2 |\dr \cdot (\nabla \times \dr)|^2 + k_3 |\dr \times (\nabla \times \dr)|^2 \\
    & + (k_2 + k_4) \left[ \mathrm{tr} (\nabla \dr)^2 - (\div \dr)^2 \right]\, ,
  \end{aligned}
\end{equation}
where the coefficients $k_1,\ k_2,\ k_3$, and $k_4$ are the measure of viscosity, depending on the material and the temperature.

The kinematic transport $\g$ is given by:
\begin{equation}\label{hat-g}
  \g_i = \lambda_1 \N_i + \lambda_2 \dr_j \A_{ij}
\end{equation}
which represents the effect of the macroscopic flow field on the microscopic structure. The following notations
\begin{equation*}
\begin{aligned}
  \A = \tfrac{1}{2}(\nabla \v + \nabla^\top\v)\,,\quad \B= \tfrac{1}{2}(\nabla \v - \nabla^\top\v)\,, \quad \N = \omega + \B \dr\,,
\end{aligned}
\end{equation*}
represent the rate of strain tensor, skew-symmetric part of the strain rate and the rigid rotation part of director changing rate by fluid vorticity, respectively. Here $\A_{ij} = \tfrac{1}{2} (\partial_j \v_i + \partial_i \v_j)$, $ \B_{ij} = \tfrac{1}{2} (\partial_j \v_i - \partial_i \v_j) $, $(\B \dr)_i =\B_{ki} \dr_k$. The material coefficients $\lambda_1$ and $\lambda_2$ reflects the molecular shape and the slippery part between the fluid and the particles. The first term of \eqref{hat-g} represents the rigid rotation of the molecule, while the second term stands for the stretching of the molecule by the flow.

The stress tensor $\sigma$ has the following form:
\begin{equation}\label{Extra-Sress-sigma}
  \begin{aligned}
    \sigma_{ji}=  \mu_1 \dr_k \A_{kp}\dr_p  \dr_i \dr_j + \mu_2  \dr_j \N_i  + \mu_3 \dr_i \N_j  + \mu_4 \A_{ij} + \mu_5 \A_{ik}\dr_k \dr_j   + \mu_6 \dr_i \A_{jk}\dr_k \,.
  \end{aligned}
\end{equation}
These coefficients $\mu_i (1 \leq i \leq 6)$ which may depend on material and temperature, are usually called Leslie coefficients, and are related to certain local correlations in the fluid. Usually, the following relations are frequently introduced in the literature.
\begin{equation}\label{Coefficients-Relations}
  \lambda_1=\mu_2-\mu_3\,, \quad\lambda_2 = \mu_5-\mu_6\,,\quad \mu_2+\mu_3 = \mu_6-\mu_5\,.
\end{equation}
The first two relations are necessary conditions in order to satisfy the equation of motion identically, while the third relation is called {\em Parodi's relation}, which is derived from Onsager reciprocal relations expressing the equality of certain relations between flows and forces in thermodynamic systems out of equilibrium. Under Parodi's relation, we see that the dynamics of an incompressible nematic liquid crystal flow involve five independent Leslie coefficients in \eqref{Extra-Sress-sigma}.

For simplicity, in this paper, we assume the external forces vanish, that is, $\mathrm{F}=0$, $\mathrm{G}=0$, and the fluid density $\rho=1$. Moreover, we take $k_1 = k_2 = k_3 = 1$, $k_4 = 0$ in \eqref{Oseen-Frank-energy}, then
$$ 2W = |\dr \cdot (\nabla \times \dr)|^2 + |\dr \times (\nabla \times \dr)|^2 + \mathrm{tr} (\nabla \dr)^2\, . $$
Since $|\dr| = 1$, this can be further simplified as $ 2W = |\nabla \dr|^2 $, which implies
\begin{equation}\label{Special-OF}
  \tfrac{\partial W}{\partial \dr_i} = 0,\  \tfrac{\partial W}{\partial(\partial_j \dr_i)} = \partial_j \dr_i\, .
\end{equation}
Taking $\beta_i=0$, then $\pi_{ij}$ reduces to
\begin{equation}\label{d-Dissipative-term}
  \partial_j \pi_{ji} = \partial_j \left( \tfrac{\partial W}{\partial(\partial_j \dr_i)} \right) = \partial_j (\partial_j \dr_i) = \Delta \dr_i\,.
\end{equation}
Thus the second evolution equation of \eqref{EL-general} is
\begin{equation}\label{d-equation-Hyper-Parab}
  \rho_1\ddot{\dr} = \Delta \dr + \gamma \dr + \lambda_1 (\dot{\dr} + \mathrm{B} \dr) + \lambda_2 \A \dr \, .
\end{equation}
Since $|\dr|=1$, it is derived from multiplying $\dr$ in \eqref{d-equation-Hyper-Parab} that
\begin{equation}\label{Lagrange-Multiplier}
  \gamma \equiv \gamma (\v, \dr, \dot{\dr}) = - \rho_1 |\dot{\dr}|^2 + |\nabla \dr|^2 - \lambda_2 \dr^\top \A \dr\, .
\end{equation}
The detailed derivation of \eqref{Lagrange-Multiplier} will be given later with deeper comments on its meanings and implications.

Combining the first equality of \eqref{Constitutive} and \eqref{Special-OF}, one can obtain that
\begin{equation}\no
    \div \hat{\sigma} = - \nabla p - \div (\nabla \dr \odot \nabla \dr)  + \div \sigma\,,
\end{equation}
where $(\nabla \dr \odot \nabla \dr)_{ij} = \sum_k \partial_i \dr_k \partial_j \dr_k\,.$ On the other hand, it can yield that by \eqref{Extra-Sress-sigma}
\begin{equation}
  \no \div \sigma= \tfrac{1}{2} \mu_4 \Delta \v + \div \tilde{\sigma}\,,
\end{equation}
where
\begin{align}
  \no \tilde{\sigma}_{ji} \equiv \big{(}\tilde{\sigma}(\v, \dr, \dot{\dr})\big{)}_{ji}  = & \mu_1 \dr_k \dr_p \A_{kp} \dr_i \dr_j + \mu_2 \dr_j (\dot{\dr}_i + \B_{ki} \dr_k)  \\
  \no &+ \mu_3 \dr_i (\dot{\dr}_j + \B_{kj} \dr_k)    +  \mu_5 \dr_j \dr_k \A_{ki} + \mu_6 \dr_i \dr_k \A_{kj} \,.
\end{align}

Hence,  Ericksen-Leslie's hyperbolic liquid crystal model reduces to the following form:
\begin{equation}\label{PHLC}
  \begin{aligned}
    \left\{ \begin{array}{c}
      \partial_t \v + \v \cdot \nabla \v - \frac{1}{2} \mu_4 \Delta \v + \nabla p = - \div (\nabla \dr \odot \nabla \dr) + \div \tilde{\sigma}\, , \\
      \div \v = 0\, ,\\
     \rho_1 \ddot{\dr} = \Delta \dr + \gamma \dr + \lambda_1 (\dot{\dr} + \B \dr) + \lambda_2 \A \dr\, ,
    \end{array}\right.
  \end{aligned}
\end{equation}
on $ \R^n \times \R^+$ with the constraint $|\dr|=1$, where the Lagrangian multiplier $\gamma$ is given by \eqref{Lagrange-Multiplier}.

In this paper, our main concern is the Cauchy problem of \eqref{PHLC} with the initial data:
\begin{equation}\label{Inital-Data}
  \v|_{t=0} = \v^{in}(x),\ \dot{\dr}|_{t=0} = {\tilde\dr}^{in}(x),\ \dr|_{t=0} = \dr^{in}(x)\,,
\end{equation}
where $\dr^{in}$ and ${\tilde\dr}^{in}$ satisfy the constraint and compatibility condition:
\begin{equation}\label{Inital-Data-Compatablity}
  |\dr^{in}|=1\,,\quad {\tilde\dr}^{in}\cdot \dr^{in}=0\,.
\end{equation}

 Let us remark that a particularly important special case of Ericksen-Leslie's model \eqref{PHLC} is that the term $\div \tilde{\sigma}$ vanishes. Namely, the coefficients $\mu_i's$, $(1\leq i \leq 6, i \neq 4)$ of $\div \tilde{\sigma}$ are chosen as 0, which immediately implies $\lambda_1 = \lambda_2 = 0$. Consequently, the system \eqref{PHLC} reduces to a model which is Navier-Stokes equations coupled with a wave map from $\mathbb{R}^n$ to $\mathbb{S}^{n-1}$:
\begin{align}\label{NS-WM}
  \left\{ \begin{array}{c}
    \partial_t \v + \v \cdot \nabla \v + \nabla p = \frac{1}{2}\mu_4 \Delta \v - \div (\nabla \dr \odot \nabla \dr )\, , \\
    \div \v =0\, ,\\
    \rho_1\ddot{\dr} = \Delta \dr + (-\rho_1 |\dot{\dr}|^2+|\nabla \dr|^2) \dr\, .
  \end{array}\right.
\end{align}

\subsection{$\rho_1=0, \lambda_1=-1$, parabolic model}

When the coefficients $\rho_1=0$ and $\lambda_1=-1$ in the third equation of \eqref{PHLC}, the system reduces to the parabolic type equations, which are also called Ericksen-Leslie's system in the literatures.  The static analogue of the parabolic Ericksen-Leslie's system is the so-called Oseen-Frank model, whose mathematical study was initialed from Hardt-Kinderlehrer-Lin \cite{Hardt-Kinderlehrer-Lin-CMP1986}. Since then there have been many works in this direction. In particular, the existence and regularity or partial regularity of the approximation (usually Ginzburg-Landau approximation as in \cite{Lin-Liu-CPAM1995}) dynamical Ericksen-Leslie's system was started by the work of Lin and Liu in \cite{Lin-Liu-CPAM1995}, \cite{Lin-Liu-DCDS1996} and \cite{Lin-Liu-ARMA2000}.

For the simplest system preserving the basic energy law
\begin{align}\label{simplified-model}
  \left\{ \begin{array}{c}
    \partial_t \v + \v \cdot \nabla \v + \nabla p =  \Delta \v - \div (\nabla \dr \odot \nabla \dr )\, , \\
    \div \v =0\, ,\\
    \partial_t \dr + \v\cdot \nabla \dr = \Delta \dr + |\nabla \dr|^2 \dr\,,\quad |\dr|=1\,,
  \end{array}\right.
\end{align}
which can be obtained by neglecting the Leslie stress and specifying some elastic constants. In 2-D case, global weak solutions with at most a finite number of singular times was proved by Lin-Lin-Wang \cite{Lin-Lin-Wang-ARMA2010}. The uniqueness of weak solutions was later on justified by Lin-Wang \cite{Lin-Wang-CAMS2010} and Xu-Zhang \cite{Xu-Zhang-JDE2012}. Recently, Lin and Wang proved global existence of weak solution for 3-D case in \cite{Lin-Wang-CPAM2016}.

For the more general parabolic Ericksen-Leslie's system,  local well-posedness is proved by Wang-Zhang-Zhang in \cite{Wang-Zhang-Zhang-ARMA2013}, and in \cite{Huang-Lin-Wang-CMP2014} regularity and existence of global solutions in $\mathbb{R}^2$ was established by Huang-Lin-Wang. The existence and uniqueness of weak solutions, also in $\mathbb{R}^2$ was proved by Hong-Xin and Li-Titi-Xin in \cite{Hong-Xin-2012} \cite{Li-Titi-Xin} respectively. Similar result was also obtained by Wang-Wang in \cite{Wang-Wang-2014}. For more complete review of the works for the parabolic Ericksen-Leslie's system, please see the reference listed above.

\subsection{$\rho_1 >0$, hyperbolic model}
If $\rho_1>0$,  \eqref{PHLC} is an incompressible Navier-Stokes equations coupled with a wave map type system for which there is very few works, comparing the corresponding parabolic model, which is Navier-Stokes coupled with a heat flow. The only notable exception might be for the most simplified model, say, in \eqref{NS-WM}, taking $\v=0$, and the spacial dimension is $1$. For this case, the system \eqref{NS-WM} can be reduced to a so-called nonlinear variational wave equation. Zhang and Zheng (later on with Bressan and others) studied systematically the dissipative and energy conservative solutions in series work starting from late 90's \cite{Zhang-Zheng-AA1998, Zhang-Zheng-ActaC1999, Zhang-Zheng-ARMA2000, Zhang-Zheng-CAMS2001, Zhang-Zheng-CPDE2001, Zhang-Zheng-PRSE2002, Zhang-Zheng-ARMA2003, Zhang-Zheng-AIPA2005, Bressan-Zhang-Zheng-ARMA2007, Zhang-Zheng-ARMA2010, Zhang-Zheng-CPAM2012, Chen-Zhang-Zheng-ARMA2013}.

For the multidimensional case, to our best acknowledgement, there was no mathematical work on the original hyperbolic Ericksen-Leslie's system \eqref{PHLC}. Very recently, De Anna and Zarnescu \cite{DeAnna-Zarnescu-2016} considered the inertial Qian-Sheng model of liquid crystals which couples a hyperbolic type equation involving a second order derivative with a forced incompressible Navier-Stokes equations. It is a system describing the hydrodynamics of nematic liquid crystals in the Q-tensor framework. They proved global well-posedness and twist-wave solutions. Furthermore, for the inviscid version of the Qian-Sheng model, in \cite{FRSZ-2016}, Feireisl-Rocca-Schimperna-Zarnescu proved a global existence of the {\em dissipative solution} which is inspired from that of incompressible Euler equation defined by P-L. Lions \cite{Lions-1996}.

It is well-known that the geometric constraint $|\dr|=1$ brings difficulties (particularly in higher order nonlinearities) on the Ericksen-Leslie's system \eqref{PHLC}, even in the parabolic case $\rho_1=0$. In \cite{Wang-Zhang-Zhang-ARMA2013}, Wang-Zhang-Zhang treated the parabolic case to remove this constraint by introducing a new formulation of parabolic version of \eqref{PHLC}. The key feature of their formulation is projecting the equation of $\dr$ (the third equation of \eqref{PHLC} with $\rho_1=0$) the space orthogonal to the direction $\mathrm{d}$. However, for the genuine hyperbolic system \eqref{PHLC} with $\rho_1 > 0$, this technique seems not work.  Indeed, projecting the third equation of the system \eqref{PHLC} into the {\em orthogonal} direction of $\dr$, we have:
\begin{equation}\label{Othogonal-Direction-Model}
  \begin{aligned}
    \left\{ \begin{array}{c}
      \partial_t \v + \v \cdot \nabla \v - \frac{1}{2} \mu_4 \Delta \v + \nabla p = - \div (\nabla \dr \odot \nabla \dr) + \div \tilde{\sigma}\, , \\
      \div \v = 0\, ,\\
     \rho_1 (\mathrm{I} - \dr \dr) \cdot \ddot{\dr} =  \lambda_1 (\dot{ \dr} + \B \dr) + (\mathrm{I} - \dr \dr) \cdot ( \Delta \dr  + \lambda_2 \A \dr)\, ,
    \end{array}\right.
  \end{aligned}
\end{equation}
where $\mathrm{I}$ denotes the $n \times n$ identical matrix. If taking $\rho_1=0$ and $\lambda_1=-1$, \eqref{Othogonal-Direction-Model} is exactly what was employed in \cite{Wang-Zhang-Zhang-ARMA2013}. But for the case $\rho_1 > 0$, the third equation in \eqref{Othogonal-Direction-Model} include only $(\mathrm{I} - \dr \dr) \cdot \ddot{\dr}$, while the {\em parallel} part of the second derivative term, i.e. $\dr\dr\cdot\ddot{\dr}$ is not included so it will have trouble on the energy estimate. In \cite{Wang-Zhang-Zhang-ARMA2013}, this is not a problem since $\rho_1=0$, the leading order term $\dot{ \dr} + \B \dr$ is automatically orthogonal to $\mathrm{d}$.

Our approach is projecting the third equation of \eqref{PHLC} to the direction parallel to $\mathrm{d}$:
$$ \rho_1 \dr \dr \cdot \ddot{\dr} = \gamma \dr + \dr \dr \cdot ( \Delta \dr  + \lambda_2 \A \dr)\,, $$
from which we can determine the Lagrangian multiplier $\gamma = - \rho_1 |\dot{\dr}|^2 + |\nabla \dr|^2 - \lambda_2 \dr^\top \A \dr$. Now, the key point is: with $\gamma$ given in this form,  if the initial data $\dr^{in},\ {\tilde\dr}^{in}$ satisfy $|\dr^{in}|=1$ and the compatibility condition $ {\tilde\dr}^{in} \cdot \dr^{in} = 0$, then for the {\em solution} to \eqref{PHLC}, the constraint $|\dr|=1$ will be {\em forced} to hold. Hence, these constraints need only be given on the initial data, while in the system \eqref{PHLC}, we do not need the constraint $|\dr| =1$ explicitly any more. These facts will be stated and proved in Section 2.

We remark that spiritually this is similar to \cite{Liu-Liu-Pego-CPAM2007} in which they used an unconstrained formulation of the Navier-Stokes equations, namely, the incompressibility is not required in the equations, but only imposed on the initial data.

\subsection{Main results.}

The fundamental role played is the $L^2$-estimate for \eqref{PHLC}, which we call the basic energy law. To ensure the basic energy law is dissipative, the Leslie coefficients should satisfy some constraints, see Section 2. Under these coefficients constraints, we prove the local well-posedness for the Ericksen-Leslie's hyperbolic liquid crystal system \eqref{PHLC}-\eqref{Inital-Data}. With an additional condition $\lambda_1 < 0$ which will provide extra damping, the global existence with small initial data can be established.  More precisely, the main results are stated in the following theorem:

\begin{theorem}\label{Main-Thm}
 Let the integer $s > \frac{n}{2}+1 $, and let the the initial data satisfy $\v^{in},\ {\tilde\dr}^{in} \in H^s(\R^n),\ \nabla \dr^{in} \in H^{s}(\R^n)$, $|\dr^{in}| = 1$, $ {\tilde\dr}^{in} \cdot \dr^{in} = 0$. The initial energy is defined as $E^{in} \equiv |\v^{in}|^2_{H^s} + \rho_1 |{\tilde\dr}^{in}|^2_{H^s} + |\nabla \dr^{in}|^2_{H^s}$. If the Leslie coefficients satisfy the relations that either
 \begin{equation}\label{Coeffs-3}
    \mu_1 \geq 0\,, \mu_4 > 0\,, \lambda_1 < 0\,, \mu_5 + \mu_6 + \tfrac{\lambda_2^2}{\lambda_1} \geq 0\,,
  \end{equation}
  or
  \begin{equation}\label{Coeffs-4}
    \mu_1 \geq 0\,, \mu_4 > 0\,, \lambda_1 = 0\,, (1 - \delta ) \mu_4 ( \mu_5 + \mu_6 ) \geq 2 | \lambda_2 |^2
  \end{equation}
  for some $ \delta \in (0,1)$, then the following statements hold:

 (I). If the initial energy $E^{in} < \infty$, then there exists $T > 0$, depending only on $E^{in}$ and Leslie coefficients, such that the system \eqref{PHLC}-\eqref{Inital-Data} admits a unique solution $\v \in L^\infty(0,T;H^s(\R^n)) \cap L^2(0,T;H^{s+1}(\R^n))$, $\nabla \dr \in L^\infty(0,T;H^{s}(\R^n))$ and $\dot{\dr} \in L^\infty(0,T;H^s(\R^n))$. Moreover, the solution $(\v, \dr)$ satisfies
  $$ \sup_{0 \leq t \leq T}\left(  |\v|^2_{H^s} + \rho_1 |\dot{\dr}|^2_{H^2} + |\nabla \dr|^2_{H^s} \right)(t) + \tfrac{1}{2} \mu_4 \int_0^T |\nabla \v|^2_{H^s} (\tau) \d \tau \leq C_0 \,,$$
  where the positive constant $C_0$ depends only on $E^{in}$, Leslie coefficients and $T$.

   (II). If in addition, $\lambda_1 < 0$, henceforth the coefficients constraints \eqref{Coeffs-3} hold, then there is a constant $\eps_0 > 0$, depending only on Leslie coefficients, such that if the initial data satisfy $E^{in}  \leq \eps_0$, then the system \eqref{PHLC}-\eqref{Inital-Data} has a unique global solution $\v \in L^\infty(0,\infty;H^s(\R^n)) \cap L^2(0,\infty;H^{s+1}(\R^n))$, $\nabla \dr \in L^\infty(0,\infty;H^{s}(\R^n))$ and $\dot{\dr} \in L^\infty(0,\infty;H^s(\R^n))$. Moreover, the solution $(\v, \dr)$ satisfies
  \begin{align}
  \no \sup_{ t \geq 0} \left( |\v|^2_{H^s} + \rho_1 |\dot{\dr}|^2_{H^s} + |\nabla \dr|^2_{H^s} \right) + \tfrac{1}{2} \mu_4 \int_0^\infty |\nabla \v|^2_{H^s} \d t
 \leq C_1 E^{in} \,,
 \end{align}
 where the constant $C_1 > 0$ depends upon the Leslie coefficients and inertia constant $\rho_1$.
\end{theorem}

\begin{remark}
  The simplified case $\mu_1 = \mu_2 = \mu_3 = \mu_5 = \mu_6 = 0$, which corresponds to the Hyperbolic Ericksen-Leslie model \eqref{NS-WM}, satisfies the Leslie's coefficients relations \eqref{Coeffs-4}. So, the local well-posedness of the simplified system \eqref{NS-WM} holds.
\end{remark}

We now sketch the main ideas of the proof of the above theorem. One of the key difficulties is that the geometric constraint $\dr=1$ is hard to be kept in the construction of the approximate system. In stead, this constraint should be put on the initial data, then it will be forced to be satisfied during the evolution if the solutions are smooth. This is  the similar treatment in Ericksen-Leslie's parabolic liquid crystal system. However, for the parabolic case, this property is kept simply by the maximal principle. For example, see \cite{Lin-Lin-Wang-ARMA2010}. Inspired by idea from the standard wave map, we consider $h(x,t)=|\dr(x,t)|^2-1$ which satisfies a nonlinear wave equation with a given velocity field $\v$. Then we can use a Gronwall type inequality to show if initially $h(0,x)=0$, then it will vanish provided that $(\v, \dr)$ are smooth and $\dr$ is bounded.

Our next step is that for a given smooth velocity field $\v(x,t)$, the equation of $\dr(x,t)$ can be solved locally in time. We use the mollifier method to construct the approximate solutions {\em without} assuming the constraint $|\dr|=1$. The key point is that we can derive the uniform energy estimate also without this constraint. Then as mentioned above, the condition $|\dr|=1$ will be automatically obeyed in the time interval that the solution is smooth.

Now we carefully design an iteration scheme to construct approximate solutions $(\v^k, \dr^k)$: solve the Stokes type equation of $\v^{k+1}$ in terms of $\v^k$ and $\dr^k$, and solve the wave map type equation of $\dr^{k+1}$ in terms of $\v^{k}$. Note that the key is that in this construction, the geometric constraint of $\dr^k$ will not be a trouble. So we can derive the uniform energy estimate by assuming $|\dr|=1$ which will significantly simply the process. We emphasize that in this step the basic energy law play a fundamental role: the main term of the higher order energy are kept in the form of the basic $L^2$-estimate. Thus the existence of the local-in-time smooth solutions can be proved.

To prove the global-in-time smooth solutions, the dissipation in the a priori estimate in the last step is not enough. We assume furthermore the coefficient $\lambda_1$ is strictly negative which will give us an additional dissipation, based on a new a priori estimate by carefully designing the energy and energy dissipation. Then we can show the existence of the global smooth solutions with the small initial data.

The organization of this paper is as follows: in the next section, we prove the basic energy law, and give the conditions on the Leslie coefficients so that the energy is dissipative. In Section \ref{Sec-APE}, we provide an a priori estimate of the system \eqref{PHLC}. In Section \ref{Sec-LagMult-Geom}, we justify the relation between the Lagrangian multiplier $\gamma$ and $|\dr|=1$, which guarantees that the condition of unit length of $\dr$ needs only be put on the initial dat. In Section \ref{Sec-LWP-u-Given}, we show the local existence of $\dr$ for a given velocity field $\v$, which will be employed in constructing the iterating approximate system of \eqref{PHLC}-\eqref{Inital-Data} in Section \ref{Sec-IAS}. In Section \ref{Sec-IAS}, we construct the approximate equation of \eqref{PHLC} by iteration. In Section \ref{Sec-LWP}, we prove the local well-posedness of \eqref{PHLC} with large initial data by obtaining  uniform energy bounds of the iterating approximate system \eqref{Iter-Appr-Syt}. In Section \ref{Sec-Global}, we globally extend the solution of \eqref{PHLC}-\eqref{Inital-Data} constructed in Section \ref{Sec-LWP} under the small initial energy condition when the {\em additional} coefficients constraint $\lambda_1 < 0$ is imposed.

\section{Basic energy law}\label{Sec-BEL}

In this section, we formally derive the basic energy law. Throughout this section, the smoothness of the solution $(\v, \dr)$ to \eqref{PHLC}-\eqref{Inital-Data} is assumed.

\begin{proposition}\label{Prop-Basic-Ener-Law}
  If $(\v, \dr)$ is a smooth solution to the system \eqref{PHLC} with the initial conditions \eqref{Inital-Data}, then the flow satisfies the following relation:
  \begin{equation}\label{Basic-Ener-Law}
    \begin{aligned}
      & \tfrac{1}{2} \tfrac{\d}{\d t} \big{(} |\v|^2_{L^2} + \rho_1 |\dot{\dr}|^2_{L^2} + |\nabla \dr|^2_{L^2} \big{)} + \tfrac{1}{2} \mu_4 |\nabla \v|^2_{L^2} + \mu_1 |\dr^\top \A \dr|^2_{L^2} \\
      & - \lambda_1 |\dot{\dr} + \B \dr|^2_{L^2} - 2 \lambda_2 \l \dot{\dr} + \B \dr , \A \dr \r + (\mu_5 + \mu_6) |\A \dr|^2_{L^2} = 0 \,.
    \end{aligned}
  \end{equation}
  Moreover, the basic energy law \eqref{Basic-Ener-Law} is dissipated if the Leslie coefficients satisfy that either the constraints \eqref{Coeffs-3}, i.e.
  \begin{equation*}
     \mu_1 \geq 0\,, \mu_4 > 0\,, \lambda_1 < 0\,, \mu_5 + \mu_6 + \tfrac{\lambda_2^2}{\lambda_1} \geq 0\,,
  \end{equation*}
  or the constraints \eqref{Coeffs-4}, i.e.
  \begin{equation*}
    \mu_1 \geq 0\,, \mu_4 > 0\,, \lambda_1 = 0\,, (1 - \delta ) \mu_4 ( \mu_5 + \mu_6 ) \geq 2 | \lambda_2 |^2
  \end{equation*}
  for some $ \delta \in (0,1)$.
\end{proposition}

\begin{proof}

  Taking $L^2$-inner product with $\v$ in the first equation of \eqref{PHLC}, we have
  \begin{equation*}
    \tfrac{1}{2} \tfrac{\d}{\d t} |\v|^2_{L^2} + \tfrac{1}{2} \mu_4 |\nabla \v|^2_{L^2} = \l - \div (\nabla \dr \odot \nabla \dr), \v \r + \l \partial_j \tilde{\sigma}_{ji} , \v_i \r \,.
  \end{equation*}
  Taking $L^2$-inner product with $\dot{\dr}$ in the third equation of \eqref{PHLC}, we obtain
  \begin{equation*}
    \tfrac{1}{2} \tfrac{\d}{\d t} \big{(} \rho_1 |\dot{\dr}|^2_{L^2} + |\nabla \dr|^2_{L^2} \big{)} - \lambda_1 |\dot{\dr}|^2_{L^2} = \l \Delta \dr, \v \cdot \nabla \dr \r + \lambda_1 \l \B_{ki} \dr_k, \dot{\dr}_i \r + \lambda_2 \l \A_{ki} \dr_k, \dot{\dr}_i \r \,.
  \end{equation*}
  Here we make use of the fact $\dot{\dr} \cdot \dr = 0$. Noticing that $ \l - \div (\nabla \dr \odot \nabla \dr), \v \r + \l \Delta \dr, \v \cdot \nabla \dr \r = 0 $, we know
  \begin{equation}\label{BEL-1}
    \begin{aligned}
      & \tfrac{1}{2} \tfrac{\d}{\d t} \big{(} |\v|^2_{L^2} + \rho_1 |\dot{\dr}|^2_{L^2} + |\nabla \dr|^2_{L^2} \big{)} + \tfrac{1}{2} \mu_4 |\nabla \v|^2_{L^2} - \lambda_1 |\dot{\dr}|^2_{L^2} \\
      & = \l \partial_j \tilde{\sigma}_{ji} , \v_i \r + \lambda_1 \l \B \dr, \dot{\dr} \r + \lambda_2 \l \A \dr, \dot{\dr} \r \,.
    \end{aligned}
  \end{equation}

  Now we compute the terms in the righthand side of the above equality \eqref{BEL-1}.

  First, we notice that
  \begin{equation}\label{BEL-2}
    \begin{aligned}
      & \l \partial_j ( \mu_1 \dr_k \dr_p \A_{kp} \dr_i \dr_j  ), \v_i \r =  - \mu_1 \l \dr_k \dr_p \A_{kp} \dr_i \dr_j , \partial_j \v_i \r \\
      & = - \mu_1 \l \dr_k \dr_p \A_{kp} \dr_i \dr_j , \A_{ij} + \B_{ij} \r = - \mu_1 |\dr^\top \A \dr|^2_{L^2}\,.
    \end{aligned}
  \end{equation}
  Here we make use of the relation $\B_{ij} = - \B_{ji}$.

  Second, by using the coefficients relation \eqref{Coefficients-Relations}, the skew symmetry of the tensor $\B$ and the symmetry of the tensor $\A$, we derive
  \begin{equation}\label{BEL-3}
    \begin{aligned}
      & \l \partial_j ( \mu_2 \dr_j \B_{ki} \dr_k +　\mu_3 \dr_i \B_{kj} \dr_k ), \v_i \r =  - \l  \mu_2 \dr_j \B_{ki} \dr_k +　\mu_3 \dr_i \B_{kj} \dr_k , \partial_j  \v_i \r \\
      = & - \l (\mu_2 - \mu_3) \dr_j \B_{ki} \dr_k + \mu_3 ( \dr_j \B_{ki} \dr_k + \dr_i \B_{kj} \dr_k ), \A_{ij} + \B_{ij} \r \\
      = & - \lambda_1 \l \dr_j \B_{ki} \dr_k, \A_{ij} + \B_{ij} \r - \mu_3 \l \dr_j \B_{ki} \dr_k + \dr_i \B_{kj} \dr_k, \A_{ij} \r \\
      = & \lambda_1 |\B \dr|^2_{L^2} - ( \lambda_1 + 2 \mu_3 ) \l B_{ki} \dr_k , \A_{ij} \dr_j \r \\
      = & \lambda_1 |\B \dr|^2_{L^2} + \lambda_2 \l \B \dr , \A \dr \r \,.
    \end{aligned}
  \end{equation}

  Third, by similar calculation in the relation \eqref{BEL-3}, we imply
  \begin{equation}\label{BEL-4}
    \begin{aligned}
       & \l \partial_j ( \mu_5 \dr_j \dr_k \A_{ki} + \mu_6 \dr_i \dr_k \A_{kj} ) , \v_i \r = - \l \mu_5 \dr_j \dr_k \A_{ki} + \mu_6 \dr_i \dr_k \A_{kj} , \partial_j \v_i \r \\
      =& - \l (\mu_5 - \mu_6) \dr_j \dr_k \A_{ki} + \mu_6 ( \dr_j \dr_k \A_{ki} + \dr_i \dr_k \A_{kj} ), \A_{ij} + \B_{ij} \r \\
      =& -\lambda_2 \l \dr_j \dr_k \A_{ki} , \A_{ij} + \B_{ij} \r - \mu_6 \l \dr_j \dr_k \A_{ki} + \dr_i \dr_k \A_{kj} , \A_{ij} \r \\
      =& -\lambda_2 |\A \d|^2_{L^2} + \lambda_2 \l \dr_k \A_{ki} , \dr_j \B_{ji} \r - 2 \mu_6 |\A \dr|^2_{L^2} \\
      =& - (\mu_5 + \mu_6 ) |\A \dr|^2_{L^2} + \lambda_2 \l \A \dr , \B \dr \r \,.
    \end{aligned}
  \end{equation}

  Finally, we again make use of the same arguments in \eqref{BEL-3} or \eqref{BEL-4} and then we yield
  \begin{equation}\label{BEL-5}
    \begin{aligned}
      & \l \partial_j ( \mu_2 \dr_j \dot{\dr}_i + \mu_3 \dr_j \dot{\dr}_i ), \v_i \r = - \l \mu_2 \dr_j \dot{\dr}_i + \mu_3 \dr_j \dot{\dr}_i , \partial_j \v_i \r \\
      =& - \l (\mu_2 - \mu_3) \dr_j \dot{\dr}_i + \mu_3 ( \dr_j \dot{\dr}_i + \dr_i \dot{\dr}_j ), \A_{ij} + \B_{ij} \r \\
      =& - \lambda_1 \l \dr_j \dot{\dr}_i, \A_{ij} + \B_{ij} \r - \mu_3 \l \dr_j \dot{\dr}_i + \dr_i \dot{\dr}_j , \A_{ij} \r \\
      =& - \lambda_1 \l \dr_j \dot{\dr}_i, \A_{ij}  \r + \lambda_1 \l \dr_j \dot{\dr}_i,  \B_{ji} \r - 2 \mu_3 \l \dr_j \dot{\dr}_i, \A_{ij}  \r \\
      =& \lambda_2 \l \dot{\dr} , \A \dr \r + \lambda_1 \l \dot{\dr} , \B \dr \r \,.
    \end{aligned}
  \end{equation}

  Therefore, by plugging the equalities \eqref{BEL-2}, \eqref{BEL-3}, \eqref{BEL-4}, \eqref{BEL-5} into \eqref{BEL-1}, we have
  \begin{equation*}
    \begin{aligned}
      & \tfrac{1}{2} \tfrac{\d}{\d t} \big{(} |\v|^2_{L^2} + \rho_1 |\dot{\dr}|^2_{L^2} + |\nabla \dr|^2_{L^2} \big{)} + \tfrac{1}{2} \mu_4 |\nabla \v|^2_{L^2} + \mu_1 |\dr^\top \A \dr|^2_{L^2} \\
      & - \lambda_1 |\dot{\dr} + \B \dr|^2_{L^2} - 2 \lambda_2 \l \dot{\dr} + \B \dr , \A \dr \r + (\mu_5 + \mu_6) |\A \dr|^2_{L^2} = 0 \,,
    \end{aligned}
  \end{equation*}
  and, consequently, we complete the proof of Proposition \ref{Prop-Basic-Ener-Law}.
\end{proof}

\section{A priori estimate}\label{Sec-APE}

In this section, we derive {\em a priori} estimate. We assume $(\v,\dr)$ is a smooth solution to the system \eqref{PHLC}-\eqref{Inital-Data}, and introduce the following energy functionals:
\begin{equation*}
  \begin{aligned}
    E(t) =& |\v|^2_{H^s} + \rho_1 |\dot{\dr}|^2_{H^s} + |\nabla \dr|^2_{H^s} \,,
  \end{aligned}
\end{equation*}
and for the case of $\lambda_1 < 0$, i.e. \eqref{Coeffs-3}
\begin{equation*}
  \begin{aligned}
    D(t) =& \tfrac{1}{2} \mu_4 |\nabla \v|^2_{H^s} + \mu_1 \sum_{k=0}^s |\dr^\top (\nabla^k \A) \dr|^2_{L^2}  - \lambda_1 \sum_{k=0}^s |\nabla^k \dot{\dr} + (\nabla^k \B) \dr + \tfrac{\lambda_2}{\lambda_1} (\nabla^k \A) \dr |^2_{L^2} \\
    & + ( \mu_5 + \mu_6 + \tfrac{\lambda_2^2}{\lambda_1}  ) \sum_{k=0}^s |(\nabla^k \A) \dr|^2_{L^2} \,.
  \end{aligned}
\end{equation*}
If the case $\lambda_1 = 0$, i.e. \eqref{Coeffs-4} holds, the energy functional $D(t)$ is defined as
\begin{equation*}
  \begin{aligned}
    D(t) =& \tfrac{1}{4} \delta \mu_4 |\nabla \v|^2_{H^s} + \mu_1 \sum_{k=0}^s |\dr^\top (\nabla^k \A) \dr|^2_{L^2}  + ( \mu_5 + \mu_6 - \tfrac{2 \lambda_2^2}{(1-\delta) \mu_4}  ) \sum_{k=0}^s |(\nabla^k \A) \dr|^2_{L^2}  \\
    & + \tfrac{1}{2} (1 - \delta) \mu_4  \sum_{k=0}^s \Big( |\nabla^{k+1} \v|_{L^2} - \tfrac{2 |\lambda_2|}{(1- \delta) \mu_4} |(\nabla^k \A) \dr|_{L^2} \Big)^2 \,.
  \end{aligned}
\end{equation*}

\begin{lemma}\label{Lm-APE}
  Let $s > \tfrac{n}{2} + 1$. Assume $(\v,\dr)$ is a smooth solution to the system \eqref{PHLC}-\eqref{Inital-Data}. Then there exists a constant $C > 0$, depending only on Leslie coefficients and inertia density constant $\rho_1$, such that
  \begin{equation*}
    \tfrac{1}{2} \tfrac{\d}{\d t} E(t) + D(t) \leq C E^2 (t) + C \sum_{p=1}^4 E^\frac{p+1}{2} (t) D^\frac{1}{2} (t) \,.
  \end{equation*}
\end{lemma}

\begin{proof}
  For all $0 \leq k \leq s$, we act $\nabla^k$ on the first equation of \eqref{PHLC} and take $L^2$-inner product with $\nabla^k \v$ and then we have
  \begin{equation*}
    \begin{aligned}
      & \tfrac{1}{2} \tfrac{\d}{\d t} |\nabla^k \v|^2_{L^2} + \tfrac{1}{2} \mu_4 |\nabla^{k+1} \v|^2_{L^2} \\
      =& - \l \nabla^k (\v \cdot \nabla \v) , \nabla^k \v \r - \l \nabla^k \div (\nabla \dr \odot \nabla \dr), \nabla^k \v \r + \l \nabla^k \div \tilde{\sigma}, \nabla^k \v \r \,.
    \end{aligned}
  \end{equation*}
  Again via acting $\nabla^k$ on the third equation of \eqref{PHLC} and taking $L^2$-inner product with $\nabla^k \dot{\dr}$, we have
  \begin{equation*}
    \begin{aligned}
      & \tfrac{1}{2} \tfrac{\d}{\d t} \big{(} \rho_1 |\nabla^k \dot{\dr}|^2_{L^2} + |\nabla^{k+1} \dr|^2_{L^2} \big{)} \\
      =& - \l \nabla^k ( \rho_1 \v \cdot \nabla \dot{\dr} ) , \nabla^k \dot{\dr} \r + \l \nabla^k \Delta \dr, \nabla^k ( \v \cdot \nabla \dr ) \r  \\
      & + \l \nabla^k (\gamma \dr), \nabla^k \dot{\dr} \r + \lambda_1 \l \nabla^k (\dot{\dr} + \B \dr ), \nabla^k \dot{\dr} \r + \lambda_2 \l \nabla^k (\A \dr), \nabla^k \dot{\dr} \r \,.
    \end{aligned}
  \end{equation*}
  Therefore, we have
  \begin{equation}\label{APE-1}
    \begin{aligned}
      & \tfrac{1}{2} \tfrac{\d}{\d t} \Big{(} |\nabla^k \v|^2_{L^2} + \rho_1 |\nabla^k \dot{\dr}|^2_{L^2} + |\nabla^{k+1} \dr|^2_{L^2} \Big{)} + \tfrac{1}{2} \mu_4 |\nabla^{k+1} \v|^2_{L^2} \\
      = & - \l \nabla^k (\v \cdot \nabla \v) , \nabla^k \v \r - \l \nabla^k ( \rho_1 \v \cdot \nabla \dot{\dr} ) , \nabla^k \dot{\dr} \r \qquad\qquad\qquad\qquad\qquad\quad I_1 \\
      & - \l \nabla^k \div (\nabla \dr \odot \nabla \dr), \nabla^k \v \r + \l \nabla^k \Delta \dr, \nabla^k ( \v \cdot \nabla \dr ) \r \qquad\qquad\qquad\quad\quad\ \ I_2 \\
      & + \l \nabla^k (\gamma \dr), \nabla^k \dot{\dr} \r \qquad\qquad\qquad\qquad\qquad\qquad\qquad\qquad\qquad\qquad\qquad\ \ I_3 \\
      & + \l \nabla^k \div \tilde{\sigma}, \nabla^k \v \r  + \lambda_1 \l \nabla^k (\dot{\dr} + \B \dr ), \nabla^k \dot{\dr} \r + \lambda_2 \l \nabla^k (\A \dr), \nabla^k \dot{\dr} \r  \ \qquad\ I_4\\
      \equiv & \ I_1 + I_2 + I_3 + I_4\,.
    \end{aligned}
  \end{equation}
  Here we divide the righthand side terms of the above equality \eqref{APE-1} into four layers. If $k=0$, the equality \eqref{APE-1} implies the basic energy law \eqref{Basic-Ener-Law} (in this case, $I_1 = I_2 = I_3 = 0$). If $1 \leq k \leq s$, the layers $I_1$, $I_2$ and $I_3$ can be controlled by the free energy terms in left hand side of \eqref{APE-1}. For the layer $I_4$, the intermediate derivatives terms can be controlled by the free energy terms in left hand side of \eqref{APE-1}, and the endpoint derivatives terms reduce to the dissipated terms with the same form in basic energy law \eqref{Basic-Ener-Law}.

  Now we estimate \eqref{APE-1} term by term for $1 \leq  k \leq s$.

  We take advantage of H\"older inequality, Sobolev embedding and the fact $\div \v = 0$, then we have
  \begin{equation}\label{APE-2}
    \begin{aligned}
      I_1 =& - \sum_{\substack{a+b=k \\ a \geq 1}} \l \nabla^a \v \nabla^{b+1} \v , \nabla^k \v \r - \rho_1 \sum_{\substack{a+b=k \\ a \geq 1}} \l \nabla^a \v \nabla^{b+1} \dot{\dr} , \nabla^k \dot{\dr} \r \\
      \lesssim & \sum_{\substack{a+b=k \\ a \geq 1}} | \nabla^a \v |_{L^4} |\nabla^{b+1} \v|_{L^4} |\nabla^k \v|_{L^2} \\
       & + \rho_1 |\nabla \v|_{L^\infty} |\nabla^k \dot{\dr}|^2_{L^2} + \rho_1 \sum_{\substack{a+b=k \\ a \geq 2}} |\nabla^a \v|_{L^4} |\nabla^{b+1} \dot{\dr}|_{L^4} |\nabla^k \dot{\dr}|_{L^2} \\
      \lesssim & |\nabla \v|_{H^s} |\v|^2_{\dot{H}^s} + \rho_1 |\nabla \v|_{H^s} |\dot{\dr}|^2_{H^s} \,.
    \end{aligned}
  \end{equation}

  For the second layer $I_2$, the endpoint derivative terms will vanish with same reason of $I_2 = 0$ in the case $k = 0$. So we estimate
    \begin{align}\label{APE-3}
      \no I_2 =& - \l \nabla^k \partial_i \dr_p \Delta \dr_p, \nabla^k \v_i \r  - \sum_{\substack{a+b=k \\ 1 \leq a \leq k -1}} \l \nabla^a \partial_i \dr_p \nabla^b \Delta \dr_p, \nabla^k \v_i \r \\
      \no & - \sum_{\substack{ a+b=k \\ 1 \leq b \leq k-1 }} \l \nabla^k \partial_j \dr_p , \nabla^a \partial_j \v_i \nabla^b \partial_i \dr_p + \nabla^a \v_i \nabla^b \partial_i \partial_j \dr_p \r - \l \nabla^k \partial_j \dr_p , \partial_j \v_i \nabla^k \partial_i \dr_p \r \\
       \lesssim & |\nabla^{k+1} \dr|_{L^2} |\Delta \dr|_{L^4} |\nabla^k \v|_{L^4}  + |\nabla^{k+1} \dr|^2_{L^2} |\nabla \v|_{L^\infty}\\
       \no & + \sum_{\substack{a+b=k \\ 1 \leq a \leq k -1}} |\nabla^b \Delta \dr_p|_{L^2} |\nabla^a \partial_i \dr_p|_{L^4} |\nabla^k \v_i|_{L^4} \\
      \no & + \sum_{\substack{a+b=k \\ 1 \leq b \leq k -1}} |\nabla^k \partial_j \dr_p|_{L^2} \big(  | \nabla^a \partial_j \v_i |_{L^4} | \nabla^b \partial_i \dr_p |_{L^4} + | \nabla^a \v_i |_{L^\infty} | \nabla^b \partial_i \partial_j \dr_p |_{L^2}   \big) \\
      \no \lesssim & |\nabla \dr|_{H^s} |\nabla \dr|_{\dot{H}^s} |\nabla \v|_{H^s} \,.
    \end{align}
  Here we make use of H\"older inequality, Sobolev embedding and the fact $\div \v = 0$.

  For the term $I_3$, we have
  \begin{equation*}
    I_3 = \l \nabla^k \gamma, \dr \cdot \nabla^k \dot{\dr} \r + \sum_{\substack{a+b=k \\ a \leq k-1}} \l \nabla^a \gamma \nabla^b \dr , \nabla \dot{\dr} \r \,.
  \end{equation*}
  Recalling the structure of Lagrangian \eqref{Lagrange-Multiplier}, for convenience of estimation, the $\gamma$ can be divided into three parts:
  \begin{equation*}
    \begin{aligned}
      & \sum_{\substack{a+b=k \\ a \leq k-1}} \l \nabla^a ( -\rho_1 |\dot{\dr}|^2 ) \nabla^b \dr , \nabla^k \dot{\dr} \r =  - \rho_1 \sum_{\substack{a+b+c=k \\ c \geq 1}} \l \nabla^a \dot{\dr} \nabla^b \dot{\dr} \nabla^c \dr , \nabla^k \dot{\dr} \r \\
      \lesssim & \rho_1 |\nabla^k \dr|_{L^4} |\dot{\dr}|_{L^\infty} |\dot{\dr}|_{L^4} |\nabla^k \dot{\dr}|_{L^2} + \rho_1 \sum_{\substack{a+b+c=k \\ 1 \leq c \leq k- 1}} |\nabla^c \dr|_{L^\infty} |\nabla^a \dot{\dr}|_{L^4} |\nabla^b \dot{\dr}|_{L^4} |\nabla^k \dot{\dr}|_{L^2} \\
      \lesssim & \rho_1 |\nabla \dr|_{H^s} |\dot{\dr}|^3_{H^s}\,,
    \end{aligned}
  \end{equation*}
  and similarly
  \begin{equation*}
    \begin{aligned}
      & \sum_{\substack{a+b=k \\ a \leq k-1}} \l \nabla^a (  |\nabla \dr|^2 ) \nabla^b \dr , \nabla^k \dot{\dr} \r \\
      = & - \l \div ( |\nabla \dr|^2 \nabla^k \dr ) , \nabla^{k-1} \dot{\dr} \r + \sum_{\substack{a+b+c=k \\ 1 \leq c \leq k-1}} \l \nabla^{a+1} \dr \nabla^{b+1} \dr \nabla^c \dr, \nabla^k \dot{\dr} \r \\
      \lesssim & |\nabla \dr|^2_{H^s} |\nabla \dr|_{\dot{H}^s} |\dot{\dr}|_{H^s} \,,
    \end{aligned}
  \end{equation*}
  and by the same arguments in the above estimation
    \begin{align*}
       & \sum_{\substack{a+b=k \\ a \leq k-1}} \l \nabla^a ( - \lambda_2 (\dr^\top \A \dr) \dr  ) \nabla^b \dr , \nabla^k \dot{\dr} \r \lesssim  |\lambda_2| \sum_{\substack{a+b+c+e=k \\ a,c,e \geq 1}} \l |\nabla^a \dr| |\nabla^{b+1} \v| |\nabla^c \dr| |\nabla^e \dr|, |\nabla^k \dot{\dr}| \r \\
      & + 3 |\lambda_2| \sum_{\substack{a+b+c=k \\ a,c \geq 1}} \l |\nabla^a \dr| |\nabla^{b+1} \v| |\nabla^c \dr| , |\nabla^k \dot{\dr}| \r  + 6 |\lambda_2| \sum_{\substack{a+b=k \\ a \geq 1}} \l |\nabla^a \dr| |\nabla^{b+1} \v| , |\nabla^k \dot{\dr}| \r \\
      & \lesssim  |\lambda_2| ( |\nabla \dr|_{H^s} + |\nabla \dr|^2_{H^s} + |\nabla \dr|^3_{H^s} + |\nabla \dr|^4_{H^s} ) |\dot{\dr}|_{H^s} |\nabla \v|_{H^s} \,.
    \end{align*}
  Thus, we have
  \begin{equation*}
    \sum_{\substack{a+b=k \\ a \leq k-1}} \l \nabla^a \gamma \nabla^b \dr , \nabla^k \dot{\dr} \r \lesssim \rho_1 |\nabla \dr|_{H^s} |\dot{\dr}|^3_{H^s} + |\nabla \dr|^2_{H^s} |\nabla \dr|_{\dot{H}^s} |\dot{\dr}|_{H^s} + |\lambda_2| \sum_{p=1}^3 |\nabla \dr|^p_{H^s}  |\dot{\dr}|_{H^s} |\nabla \v|_{H^s}\,.
  \end{equation*}
  Notice that $ \l \nabla^k \gamma, \dr \cdot \nabla^k \dot{\dr} \r \lesssim |\nabla^k \gamma|_{L^2} | \dr \cdot \nabla^k \dot{\dr} |_{L^2} $. Since $\dr \cdot \dot{\dr} = 0$, we have
   $$ \dr \cdot \nabla^k \dot{\dr} = - \sum\limits_{\substack{p+q=k \\ q \geq 1}} \nabla^p \dot{\dr} \nabla^q \dr \,.$$
   Then
  \begin{equation*}
    \begin{aligned}
      | \dr \cdot \nabla^k \dot{\dr} |_{L^2} \lesssim \sum_{\substack{p+q=k \\ q \geq 1}} |\nabla^p \dot{\dr}|_{L^4} |\nabla^q \dr|_{L^4} \lesssim |\dot{\dr}|_{H^s} |\nabla \dr|_{H^s} \,,
    \end{aligned}
  \end{equation*}
  and
  \begin{equation*}
    \begin{aligned}
      |\nabla^k \gamma|_{L^2} \lesssim & \rho_1 \sum_{a+b=k} |\nabla^a \dot{\dr} \nabla^b \dot{\dr}|_{L^2} + \sum_{a+b=k} |\nabla^{a+1} \dr \nabla^{b+1} \dr|_{L^2}  + |\lambda_2| \sum_{a+b+c=k} |\nabla^a \dr^\top \nabla^b \A \nabla^c \dr|_{L^2} \\
      \lesssim & \rho_1 |\dot{\dr}|^2_{H^s} + |\nabla \dr|_{H^s} |\nabla \dr|_{\dot{H}^s} + |\lambda_2| \sum_{p=0}^2 |\nabla \dr|^p_{H^s} |\nabla \v|_{H^s}\,.
    \end{aligned}
  \end{equation*}
  Therefore, we estimate the term $I_3$ as
  \begin{equation}\label{APE-4}
    I_3 \lesssim \rho_1 |\nabla \dr|_{H^s} |\dot{\dr}|^3_{H^s} + |\nabla \dr|^2_{H^s} |\nabla \dr|_{\dot{H}^s} |\dot{\dr}|_{H^s} + |\lambda_2| \sum_{p=1}^3 |\nabla \dr|^p_{H^s}  |\dot{\dr}|_{H^s} |\nabla \v|_{H^s}\,.
  \end{equation}

  The term $I_4$ can be rewritten as the endpoint derivative terms part $I_4^{ep}$ and the intermediate derivative terms part $I_4^{im}$, hence
  $$ I_4 = I_4^{ep} + I_4^{im} \,, $$
  where
  \begin{equation*}
    \begin{aligned}
      I_4^{ep} = & - \mu_1 \l  \dr_p \dr_q \nabla^k \A_{pq} \dr_i \dr_j  , \nabla^k \partial_j \v_i \r - \l \mu_2 \dr_j \nabla^k \dot{\dr}_i + \mu_3 \dr_i \nabla^k \dot{\dr}_j , \nabla^k \partial_j \v_i \r \\
      & - \l \mu_2 \dr_j \nabla^k \B_{pi} \dr_p + \mu_3 \dr_i \nabla^k \B_{pj} \dr_p , \nabla^k \partial_j \v_i \r \\
      & - \l \mu_5 \dr_j \dr_p \nabla^k \A_{pi} + \mu_6 \dr_i \dr_p \nabla^k \A_{pj}, \nabla^k \partial_j \v_i \r \\
      & + \lambda_1 |\nabla^k \dot{\dr}|^2_{L^2} + \lambda_1 \l (\nabla^k \B ) \dr , \nabla^k \dot{\dr} \r + \lambda_2 \l (\nabla^k \A ) \dr , \nabla^k \dot{\dr} \r \,,
    \end{aligned}
  \end{equation*}
  and
  \begin{equation*}
    \begin{aligned}
      I_4^{im} = & - \mu_1 \sum_{\substack{a+b=k \\ a \geq 1}} \l \nabla^a (\dr_p \dr_q \dr_i \dr_j) \nabla^b \A_{pq} , \nabla^k \partial_j \v_i \r \\
      & - \sum_{\substack{a+b=k \\ a \geq 1}} \l \mu_2 \nabla^a \dr_j \nabla^b \dot{\dr}_i + \mu_3 \nabla^a \dr_i \nabla^b \dot{\dr}_j , \nabla^k \partial_j \v_i \r \\
      & - \sum_{\substack{a+b=k \\ a \geq 1}} \l \mu_2 \nabla^a ( \dr_j \dr_p ) \nabla^b \B_{pi} + \mu_3 \nabla^a ( \dr_i \dr_p ) \nabla^b \B_{pj} , \nabla^k \partial_j \v_i \r \\
      & - \sum_{\substack{a+b=k \\ a \geq 1}} \l \mu_5 \nabla^a (\dr_j \dr_p) \nabla^b \A_{pi} + \mu_6 \nabla^a (\dr_i \dr_p) \nabla^b \A_{pj} , \nabla^k \partial_j \v_i \r \\
      & + \sum_{\substack{a+b=k \\ b \geq 1}} \l \lambda_1 \nabla^a \B \nabla^b \dr + \lambda_2 \nabla^a \A \nabla^b \dr , \nabla^k \dot{\dr} \r \\
      \equiv & \ I_{41}^{im} + I_{42}^{im} + I_{43}^{im} + I_{44}^{im} + I_{45}^{im}\,.
    \end{aligned}
  \end{equation*}
  According to the same arguments of deriving the basic energy law in Section \ref{Sec-BEL}, one can calculate the endpoint derivative terms part $I_4^{ep}$ as 
  \begin{equation}\label{APE-5}
    \begin{aligned}
      I_4^{ep} =& - \mu_1 |\dr^\top (\nabla^k \A) \dr|^2_{L^2} + \lambda_1 |\nabla^k \dot{\dr} + (\nabla^k \B) \dr + \tfrac{\lambda_2}{\lambda_1} (\nabla^k \A) \dr |^2_{L^2} \\
      & - ( \mu_5 + \mu_6 + \tfrac{\lambda_2^2}{\lambda_1} ) |(\nabla^k \A) \dr|^2_{L^2}
    \end{aligned}
  \end{equation}
 for the case $\lambda_1 < 0$, i.e. \eqref{Coeffs-3}, while for the case $\lambda_1 = 0$, i.e. the relations \eqref{Coeffs-4} it is
  \begin{equation}\label{APE-5*}
    \begin{aligned}
      I_4^{ep} =& - \mu_1 |\dr^\top (\nabla^k \A) \dr|^2_{L^2} - \tfrac{1}{2} (1-\delta) \mu_4 \big( |\nabla^{k+1} \v|_{L^2} - \tfrac{2 |\lambda_2|}{(1-\delta) \mu_4} | ( \nabla^k \A ) \dr |_{L^2} \big)^2 \\
      & + \tfrac{1}{2} (1-\delta) \mu_4 |\nabla^{k+1} \v|^2_{L^2} - ( \mu_5 + \mu_6 - \tfrac{2 \lambda_2^2}{(1-\delta) \mu_4} ) |(\nabla^k \A) \dr|^2_{L^2} \,.
    \end{aligned}
  \end{equation}
  It remains to estimate the intermediate derivative terms part $I_4^{im} $. We make use of H\"older inequality, Sobolev embedding and the fact $|\dr| = 1$ to estimate it term by term:
  \begin{equation*}
    \begin{aligned}
      I_{41}^{im} \lesssim & \mu_1 \sum_{\substack{a+b=k \\ a \geq 1}} \sum_{\substack{a_1+a_2+a_3+a_4=a \\ a_1,a_2,a_3,a_4 \geq 1 }} \l |\nabla^{a_1} \dr| |\nabla^{a_2} \dr| |\nabla^{a_3} \dr| |\nabla^{a_4} \dr| |\nabla^{b+1} \v|, |\nabla^{k+1} \v| \r \\
      & + 4 \mu_1 \sum_{\substack{a+b=k \\ a \geq 1}} \sum_{\substack{a_1+a_2+a_3 =a \\ a_1,a_2,a_3 \geq 1 }} \l |\nabla^{a_1} \dr| |\nabla^{a_2} \dr| |\nabla^{a_3} \dr| |\nabla^{b+1} \v|, |\nabla^{k+1} \v| \r \\
      & + 12 \mu_1 \sum_{\substack{a+b=k \\ a \geq 1}} \sum_{\substack{a_1+a_2 =a \\ a_1,a_2  \geq 1 }} \l |\nabla^{a_1} \dr| |\nabla^{a_2} \dr| |\nabla^{b+1} \v|, |\nabla^{k+1} \v| \r \\
       & + 24 \mu_1 \sum_{\substack{a+b=k \\ a \geq 1}} \l |\nabla^{a} \dr| |\nabla^{b+1} \v|, |\nabla^{k+1} \v| \r \\
      \lesssim & \mu_1 \big{(} |\nabla \dr|_{H^s} + |\nabla \dr|^2_{H^s} + |\nabla \dr|^3_{H^s} + |\nabla \dr|^4_{H^s} \big{)} |\v|_{\dot{H}^s} |\nabla \v|_{H^s} \,,
    \end{aligned}
  \end{equation*}
  and
  \begin{equation*}
    \begin{aligned}
      I_{42}^{im} \lesssim & ( |\mu_2| + |\mu_3| ) \sum_{\substack{a+b=k \\ a \geq 1}} | \nabla^a \dr|_{L^4} |\nabla^b \dot{\dr}|_{L^4} |\nabla^{k+1} \v|_{L^2} \\
      \lesssim & ( |\mu_2| + |\mu_3| ) |\nabla \dr|_{H^s} |\dot{\dr}|_{H^s} |\nabla \v|_{H^s}\,,
    \end{aligned}
  \end{equation*}
  and
  \begin{equation*}
    \begin{aligned}
      I_{43}^{im} = & - \sum_{\substack{a+b=k \\ a \geq 1}} \sum_{a_1 + a_2 =a} \l \mu_2 \nabla^{a_1} \dr_j \nabla^{a_2} \dr_p \nabla^b \B_{pi} + \mu_3 \nabla^{a_1} \dr_i \nabla^{a_2} \dr_p \nabla^b \B_{pj} , \nabla^k \partial_j \v_i \r \\
      \lesssim & ( |\mu_2| + |\mu_3| ) \sum_{\substack{a+b=k \\ a \geq 1}} \l |\nabla^a \dr| |\nabla^{b+1} \v|, |\nabla^{k+1} \v| \r \\
      & + ( |\mu_2| + |\mu_3| ) \sum_{\substack{a_1+a_2+b=k \\ a_1,a_2 \geq 1}} \l |\nabla^{a_1} \dr| |\nabla^{a_2} \dr| |\nabla^{b+1} \v|, |\nabla^{k+1} \v| \r \\
      \lesssim & ( |\mu_2| + |\mu_3| ) ( |\nabla \dr|_{H^s} + |\nabla \dr|^2_{H^s} ) |\v|_{\dot{H}^s} |\nabla \v|_{H^s} \,,
    \end{aligned}
  \end{equation*}
  and similarly
  \begin{equation*}
    \begin{aligned}
      I_{44}^{im} \lesssim ( |\mu_5| + |\mu_6| ) ( |\nabla \dr|_{H^s} + |\nabla \dr|^2_{H^s} ) |\v|_{\dot{H}^s} |\nabla \v|_{H^s} \,,
    \end{aligned}
  \end{equation*}
  and
  \begin{equation*}
    \begin{aligned}
      I_{45}^{im} \lesssim & ( |\lambda_1| + |\lambda_2| ) \sum_{\substack{a+b=k \\ b \geq 1}} \l |\nabla^{a+1} \v| |\nabla^b \dr|, |\nabla^k \dot{\dr}| \r \\
      \lesssim & ( |\lambda_1| + |\lambda_2| ) |\nabla \dr|_{H^s} |\dot{\dr}|_{H^s} |\nabla \v|_{H^s}\,.
    \end{aligned}
  \end{equation*}
  Thus, we have estimated
  \begin{equation}\label{APE-6}
    \begin{aligned}
      I_{4}^{im} \lesssim ( \mu_1 + |\mu_2| + |\mu_3| + |\mu_5| + |\mu_6| ) \sum_{p=1}^4 |\nabla \dr|^p_{H^s} ( |\v|_{\dot{H}^s} + |\dot{\dr}|_{H^s}  ) |\nabla \v|_{H^s}\,.
    \end{aligned}
  \end{equation}

  Therefore, if we plug the inequalities \eqref{APE-2}, \eqref{APE-3}, \eqref{APE-4}, \eqref{APE-5} and \eqref{APE-6} into \eqref{APE-1} and sum up for all $0 \leq k \leq s$, we obtain the inequality
  \begin{equation}\label{APE-Hs}
    \begin{aligned}
      & \tfrac{1}{2} \tfrac{\d}{\d t} \Big{(} |\v|^2_{H^s} + \rho_1 |\dot{\dr}|^2_{H^s} + |\nabla \dr|^2_{H^s} \Big{)}  + \mu_1 \sum_{k=0}^s |\dr^\top (\nabla^k \A) \dr|^2_{L^2} + \tfrac{1}{2} \mu_4 |\nabla \v|^2_{H^s} \\
      & - \lambda_1 \sum_{k=0}^s |\nabla^k \dot{\dr} + (\nabla^k \B) \dr + \tfrac{\lambda_2}{\lambda_1} (\nabla^k \A) \dr |^2_{L^2} + ( \mu_5 + \mu_6 + \tfrac{\lambda_2^2}{\lambda_1}  ) \sum_{k=0}^s |(\nabla^k \A) \dr|^2_{L^2} \\
      \lesssim & \big{(} |\v|^2_{\dot{H}^s} + \rho_1 |\dot{\dr}|^2_{H^s} + |\nabla \dr|^2_{\dot{H}^s} \big{)} |\nabla \v|_{H^s}  + \rho_1 |\dot{\dr}|^3_{H^s} |\nabla \dr|_{H^s} + |\dot{\dr}|_{H^s} |\nabla \dr|_{\dot{H}^s} |\nabla \dr|^2_{H^s} \\
      & + ( \mu_1 + |\mu_2| + |\mu_3| + |\mu_5| + |\mu_6| ) \sum_{p=1}^4 |\nabla \dr|^p_{H^s} (  |\v|_{\dot{H}^s} + |\dot{\dr}|_{H^s} ) |\nabla \v|_{H^s}
    \end{aligned}
  \end{equation}
  for the case $\lambda_1 < 0$, i.e. \eqref{Coeffs-3}, while for the case $\lambda_1 = 0$, i.e. \eqref{Coeffs-4}, substituting the inequalities \eqref{APE-2}, \eqref{APE-3}, \eqref{APE-4}, \eqref{APE-5*} and \eqref{APE-6} into \eqref{APE-1} implies an inequality, which is yielded by replacing the last three terms in the left-hand side of the inequality \eqref{APE-Hs} with 
  \begin{equation*}
    \begin{aligned}
      & \tfrac{1}{2} \delta \mu_4 |\nabla \v |^2_{H^s} + ( \mu_5 + \mu_6 - \tfrac{2 \lambda_2^2}{ (1-\delta) \mu_4 } ) \sum_{k=0}^s |( \nabla^k \A ) \dr|^2_{L^2} \\
      & + \tfrac{1}{2} ( 1 - \delta ) \mu_4 \sum_{k=0}^s \big( |\nabla^{k+1} \v|_{L^2} - \tfrac{2 |\lambda_2| }{ (1 - \delta) \mu_4 } | (\nabla^k \A) \dr |_{L^2} \big)^2 \,.
    \end{aligned}
  \end{equation*}
  Then we complete the proof of Lemma \ref{Lm-APE}.

\end{proof}

\section{Lagrangian multiplier $\gamma$ and constraint $|\dr|=1$}\label{Sec-LagMult-Geom}
In this section, we prove the following Lemma on the relation between the Lagrangian multiplier $\gamma$ and the geometric constraint $|\dr|=1$.

\begin{lemma}\label{Parallel-d-Lemma}
  Assume $(\v, \dr)$ is a classical solution to the Ericksen-Leslie's hyperbolic system \eqref{PHLC}-\eqref{Inital-Data} satisfying $\v \in L^\infty(0,T;H^s(\R^n)) \cap L^2(0,T;H^{s+1}(\R^n))$, and $\nabla \dr \in L^\infty(0,T;H^{s}(\R^n))$, $\dot{\dr} \in L^\infty(0,T;H^s(\R^n))$ and $|\dr|_{L^\infty([0,T] \times \R^n)} < \infty$ for some $T \in (0, \infty)$, where $s > \frac{n}{2} + 1$.

  If the constraint $|\dr|=1$ is required, then the Lagrangian multiplier $\gamma$ is
  \begin{equation}\label{Lagrange-Multiplier1}
    \gamma = - \rho_1 |\dot{\dr}|^2 + |\nabla \dr|^2 - \lambda_2 \dr^\top \A \dr \, .
  \end{equation}

  Conversely, if we give the form of $\gamma$ as \eqref{Lagrange-Multiplier1} and $\dr$ satisfies the initial data conditions ${\tilde\dr}^{in} \cdot \dr^{in} = 0$, $|\dr^{in}|=1$, then $ |\dr| = 1 \, . $
\end{lemma}

\begin{remark}
  In fact, if we assume that the liquid crystal flow with a given bulk velocity $\v \in L^1 (0,T; W^{1,\infty} (\R^n) )$ ( i.e. the form of the third equation of \eqref{PHLC} ) with the last two initial conditions of \eqref{Inital-Data} has a solution $\dr$ satisfying $\nabla \dr \in L^\infty(0,T;H^{s}(\R^n))$, $\dot{\dr} \in L^\infty(0,T;H^s(\R^n))$ and $|\dr|_{L^\infty([0,T] \times \R^n)} < \infty$ for some $T \in (0, \infty)$, then the conclusions of Lemma \ref{Parallel-d-Lemma} also hold.
\end{remark}

\begin{proof}
 We multiply $\dr$ in the third equation of the system \eqref{PHLC} and then we get
  \begin{align}\label{Lagrange-Multiplier-Calculate}
    \no \gamma |\dr|^2 =& \rho_1 \ddot{\dr} \cdot \dr - \Delta \dr \cdot \dr - \lambda_1 \dot{\dr} \cdot \dr + \lambda_1 \dr^\top \B \dr - \lambda_2 \dr^\top \A \dr \\
    \no =& \rho_1 (\partial_t + \v \cdot \nabla) \dot{\dr} \cdot \dr - \div \nabla \dr \cdot \dr - \lambda_1 (\partial_t + \v \cdot \nabla) \dr \cdot \dr - \lambda_2 \dr^\top \A \dr \\
    \no =& \rho_1 (\partial_t + \v \cdot \nabla) (\dot{\dr} \cdot \dr) - \rho_1 \dot{\dr} \cdot (\partial_t + \v \cdot \nabla) \dr - \div (\nabla \dr \cdot \dr) + |\nabla \dr|^2 \\
    \no & - \lambda_1 (\partial_t + \v \cdot \nabla) (\tfrac{1}{2} |\dr|^2) - \lambda_2 \dr^\top \A \dr \\
    \no =& \rho_1 (\partial_t + \v \cdot \nabla)^2 (\tfrac{1}{2} |\dr|^2) - \lambda_1 (\partial_t + \v \cdot \nabla) (\tfrac{1}{2} |\dr|^2) - \Delta (\tfrac{1}{2} |\dr|^2) \\
    \no &- \rho_1 |\dot{\dr}|^2 + |\nabla \dr|^2 - \lambda_2 \dr^\top \A \dr \, .
  \end{align}
  If $|\dr| =1$, the above equation reduces to
$$ \gamma = - \rho_1 |\dot{\dr}|^2 + |\nabla \dr|^2 - \lambda_2 \dr^\top \A \dr \, . $$

Conversely, if we give $ \gamma = - \rho_1 |\dot{\dr}|^2 + |\nabla \dr|^2 - \lambda_2 \dr^\top \A \dr $, from the above calculation and the initial conditions we have
\begin{equation}\label{|d|-Quasilinear-Hyperbolic-Equation}
\begin{aligned}
  \left\{ \begin{array}{l}
    \rho_1 (\partial_t + \v \cdot \nabla)^2 (  |\dr|^2 - 1 ) - \lambda_1 (\partial_t + \v \cdot \nabla) ( |\dr|^2 - 1 )  - \Delta ( |\dr|^2 - 1 ) = 2 \gamma (|\dr|^2 - 1) \, , \\
     (\partial_t + \v \cdot \nabla) (|\dr|^2-1) \big{|}_{t=0} = 2 {\tilde\dr}^{in} \cdot \dr^{in} = 0 \, ,\\
    (|\dr|^2 - 1 ) \big{|}_{t=0} = |\dr^{in}|^2 - 1 = 0 \, .
  \end{array} \right.
\end{aligned}
\end{equation}

Let $h= |\dr|^2 - 1$. Then $h$ solves the following Cauchy problem for a given smooth vector field $\v$:
\begin{equation}\label{h-Quasilinear-Hyperbolic-Equation}
\begin{aligned}
  \left\{ \begin{array}{l}
    \rho_1 \ddot{h} - \lambda_1 \dot{h}  - \Delta h = 2 \gamma h \, , \\
     \dot{h}|_{t=0} =  0 \, ,\\
    h|_{t=0} = 0 \, .
  \end{array} \right.
\end{aligned}
\end{equation}

 Our goal is to verify $h (t,x) = 0$ for all times $t$. Noticing that $|\dr|_{L^\infty([0,T] \times \R^n)} < \infty$ and $\v \in L^\infty(0,T;H^s(\R^n)) \cap L^2(0,T;H^{s+1}(\R^n))$ and $\dot{\dr}\, , \nabla \dr \in L^\infty(0,T;H^{s}(\R^n))$ for $s > \frac{n}{2} + 1$, we deduce that by Sobolve embedding $$|\v|_{L^1(0,T,L^\infty(\R^n))} + |\gamma|_{L^1(0,T,L^\infty(\R^n))} < \infty \, .$$

  We denote by $Z(t,x) \equiv \lambda_1 \dot{h} (t,x) + 2 \gamma (t, x) h(t,x)$. Multiplying by $\dot{h}$ in the equation \eqref{h-Quasilinear-Hyperbolic-Equation} and integrating by parts over $\R^n$, we have
  \begin{align}
    \no \tfrac{1}{2} \tfrac{\d }{\d t} ( \rho_1 |\dot{h}|_{L^2}^2 + |\nabla h|^2_{L^2} ) =&  \l Z , \dot{h} \r - \l \nabla h, \nabla \v \nabla h \r \leq |Z|_{L^2} |\dot{h}|_{L^2} + |\nabla \v|_{L^\infty} |\nabla h|^2_{L^2} \\
    \no \leq& \Big{(} \tfrac{1}{\sqrt{\rho_1}}  |Z|_{L^2} + |\nabla \v|_{L^\infty} |\nabla h|_{L^2}  \Big{)} \sqrt{ \rho_1 |\dot{h}|_{L^2}^2 + |\nabla h|^2_{L^2} }\, ,
  \end{align}
  which implies that
  $$ \tfrac{\d}{\d t} \sqrt{ \rho_1 |\dot{h}|_{L^2}^2 + |\nabla h|^2_{L^2} } \leq \tfrac{1}{\sqrt{\rho_1}} |Z|_{L^2} + |\nabla \v|_{L^\infty} |\nabla h|_{L^2} \, . $$
Then by integrating on $[0,t]$ we have
\begin{equation}\label{G-1}
  \begin{aligned}
    & \sqrt{ \rho_1 |\dot{h} (t, \cdot)|_{L^2}^2 + |\nabla h (t, \cdot)|^2_{L^2} } \leq  \sqrt{ \rho_1 |\dot{h} (0, \cdot)|_{L^2}^2 + |\nabla h (0, \cdot)|^2_{L^2} }  \\
    & + \tfrac{1}{\sqrt{\rho_1}} \int_0^t |Z(\tau, \cdot)|_{L^2} \d \tau + \int_0^t |\nabla \v (\tau, \cdot)|_{L^\infty} |\nabla h (\tau, \cdot)|_{L^2} \d \tau \, .
  \end{aligned}
\end{equation}

Let $G(t) \equiv \sqrt{ \rho_1 |\dot{h} (t, \cdot)|_{L^2}^2 + |\nabla h (t, \cdot)|^2_{L^2} } + |h(t, \cdot)|_{L^2} $. One notices that
\begin{align}\label{G-2}
  \no |Z(t, \cdot)|_{L^2} \leq& \Big{(} \tfrac{|\lambda_1|}{\sqrt{\rho_1}} + 2 |\gamma(t, \cdot)|_{L^\infty} \Big{)} \Big{(} \sqrt{\rho_1} |\dot{h} (t, \cdot)|_{L^2} + |h (t, \cdot)|_{L^2} \Big{)} \\
  \leq& \Big{(} \tfrac{|\lambda_1|}{\sqrt{\rho_1}} + 2 |\gamma(t, \cdot)|_{L^\infty} \Big{)} G(t) \, ,
\end{align}
and by the relation $\partial_t h = \dot{h} - \v \cdot \nabla h$
\begin{align}\label{G-3}
  \no |h (t, \cdot)|_{L^2} \leq & |h (0, \cdot)|_{L^2} + \int_0^t |\partial_t h(\tau, \cdot)|_{L^2} \d \tau \\
  \no \leq & |h (0, \cdot)|_{L^2} + \int_0^t ( |\dot{ h} (\tau, \cdot)|_{L^2} + |\v(\tau, \cdot)|_{L^\infty(\R^n)} |\nabla h (\tau, \cdot)|_{L^2} ) \d \tau \\
  \leq& |h(0, \cdot)|_{L^2} +  \int_0^t \big{(} \frac{1}{\sqrt{\rho_1}} + |\v(\tau, \cdot)|_{L^\infty(\R^n)} \big{)} G(\tau) \d \tau \, .
\end{align}
According  to the inequalities \eqref{G-1}, \eqref{G-2} and \eqref{G-3}, we observe that for all $0 \leq t \leq T$
$$ G(t) \leq G(0) + \int_0^t R(\tau) G(\tau) \d \tau \, , $$
where $R(t) = \frac{\sqrt{\rho_1}+|\lambda_1|  }{\rho_1} + \frac{ 2} {\sqrt{\rho_1}} |\gamma(t, \cdot)|_{L^\infty} + |\nabla \v (t, \cdot)|_{L^\infty}+  |\v(t, \cdot)|_{L^\infty} \in L^1 ([0,T])$. Then it is derived from Gronwall inequality and the fact $G(0) = 0$ that
$$ 0 \leq G(t) \leq G(0) \exp \left( \int_0^t R(\tau) \d \tau \right) = 0  $$
holds for all $t \in [0,T]$. Consequently, $h(t,x) = 0$ holds for all times $t$ and then the proof of Lemma \ref{Parallel-d-Lemma} is finished.

\end{proof}

\section{Well-posedness for a given velocity field}\label{Sec-LWP-u-Given}

In this section, we aim mainly at justifying the well-posedness of the following wave map type system with a given velocity field $\v(t,x)\in \R^n$:
\begin{equation}\label{WM}
  \left\{
    \begin{array}{c}
      \rho_1 \ddot{\dr} = \Delta \dr + \gamma(\v, \dr, \dot{\dr}) \dr + \lambda_1 (\dot{\dr} - \B \dr) + \lambda_2 \A \dr \, ,\\
      \dr \in \mathbb{S}^{n-1} \, ,
    \end{array}
  \right.
\end{equation}
with the initial conditions
\begin{equation}\label{IC-WM}
  \dr (0,x) = \dr^{in} (x) \,, \ \dot{\dr} (0,x) = \tilde{\dr}^{in} (x)\,,
\end{equation}
where the symbol $\dot{\dr} = \partial_t \dr + \v \cdot \nabla \dr$ is the first order material derivative of the vector field $\dr$ with respect to the bulk velocity $\v$, and the Lagrangian multiplier $\gamma (\v, \dr, \dot{\dr})$ is of the form of \eqref{Lagrange-Multiplier}, and the initial data satisfy the compatibility
$$ |\dr^{in}|=1 \,, \ \dr^{in} \cdot \tilde{\dr}^{in} = 0 \,. $$

More precisely, the results of well-posedness of the system \eqref{WM}-\eqref{IC-WM}, which in fact will be used in constructing the iterating approximate system of the system \eqref{PHLC}-\eqref{Inital-Data}, are stated as follows:

\begin{proposition}\label{Prop-WellPosed-WM}
  For $s > \frac{n}{2} + 1$ and $T_0 > 0$, let vector fields $(\dr^{in}, \tilde{\dr}^{in}, \v ) \in \mathbb{S}^{n-1} \times \R^n \times \R^n$ satisfy $\nabla \dr^{in} \in H^s$ , $\tilde{\dr}^{in} \in H^s$ and $\v \in  L^1 (0,T_0; H^{s+1})$. Then there exists a number $ 0 < T  \leq T_0$, depending only on $ \dr^{in}$, $\tilde{\dr}^{in}$ and $\v $, such that the system \eqref{WM}-\eqref{IC-WM} has a unique classical solution $\dr$ satisfying $\nabla \dr \,, \ \dot{\dr} \in C (0,T; H^s)$. Moreover, there is a positive constant $C_2$, depending only on $ \dr^{in}$, $\tilde{\dr}^{in}$ and $\v $, such that the solution $\dr$ satisfies the following bound
  \begin{equation*}
   \rho_1 |\dot{\dr}|^2_{L^\infty(0,T; H^s)} + |\nabla \dr|^2_{L^\infty(0,T;H^s)} \leq C_2 \,.
  \end{equation*}
\end{proposition}

\begin{proof}
  We justify this proposition by dividing into three steps. We first construct an approximate system of the system \eqref{WM}-\eqref{IC-WM} by standard mollifier methods. Second, we derive a uniform energy bounds. In the end, we take limit in the constructed approximate system by using the results of compactness.

  {\em Step 1. Construct the approximate system.} We first define the mollifier operator $\mathcal{J}_\eps$ as
  \begin{equation*}
    \mathcal{J}_\eps f = \mathcal{F}^{-1} \big{(} \mathbf{1}_{|\xi| \leq \frac{1}{\eps}} \mathcal{F} (f) \big{)} \,,
  \end{equation*}
  where the operator $\mathcal{F}$ is the standard Fourier transform and $\mathcal{F}^{-1}$ is the inverse Fourier transform. It is easy to see that the mollifier operator $\mathcal{J}_\eps$ has the property $\mathcal{J}_\eps^2 = \mathcal{J}_\eps$. We construct the following approximate system of \eqref{WM}-\eqref{IC-WM}
  \begin{equation}\label{Approx-Syst-WM-1}
    \left\{
      \begin{array}{l}
        \rho_1 \partial_t \dot{\dr}^\eps = - \rho_1 \mathcal{J}_\eps ( \v \cdot \nabla \mathcal{J}_\eps \dot{\dr}^\eps ) + \Delta \mathcal{J}_\eps \dr^\eps + \mathcal{J}_\eps \big{(} \gamma ( \v, \mathcal{J}_\eps \dr^\eps, \mathcal{J}_\eps \dot{\dr}^\eps ) \mathcal{J}_\eps \dr^\eps  \big{)} \\
        \qquad\qquad\qquad + \lambda_1 \mathcal{J}_\eps \dot{\dr}^\eps - \lambda_1 \mathcal{J}_\eps ( \B \mathcal{J}_\eps \dr^\eps ) + \lambda_2 \mathcal{J}_\eps ( \A \mathcal{J}_\eps \dr^\eps ) \,, \\
        \partial_t \dr^\eps = \dot{\dr}^\eps - \mathcal{J}_\eps ( \v \cdot \nabla \mathcal{J}_\eps \dr^\eps )\,, \\
        ( \dr^\eps, \dot{\dr}^\eps ) |_{t=0} = ( \mathcal{J}_\eps \dr_0, \mathcal{J}_\eps \tilde{\dr}_0 )\,.
      \end{array}
    \right.
  \end{equation}
  By ODE theory, we can prove that there is a maximal time interval $T_\eps > 0$, depending only on $\dr_0$, $\tilde{\dr}_0$, $\v$ and $T_0$, such that the approximate system \eqref{Approx-Syst-WM-1} admits a unique solution $\dr^\eps \in C([0,T_\eps);H^{s+1})$ and $\dot{\dr}^\eps \in C([0,T_\eps);H^{s})$. We figure out that $T_\eps \leq T_0$ for all $\eps > 0$, which is determined by the regularity of $\v$. From the fact $\mathcal{J}_\eps^2 = \mathcal{J}_\eps$, we observe that $(\mathcal{J}_\eps \dr^\eps, \mathcal{J}_\eps \dot{\dr}^\eps)$ is also a solution to the approximate system \eqref{Approx-Syst-WM-1}. Thus the uniqueness immediately implies the relation
  \begin{equation}\label{Mollify-Indent}
    (\mathcal{J}_\eps \dr^\eps, \mathcal{J}_\eps \dot{\dr}^\eps) = ( \dr^\eps, \dot{\dr}^\eps) \,.
  \end{equation}
  Therefore, from the relation \eqref{Mollify-Indent}, the solution $( \dr^\eps, \dot{\dr}^\eps)$ also solves the following system
  \begin{equation}\label{Approx-Syst-WM-2}
    \left\{
      \begin{array}{l}
        \rho_1 \partial_t \dot{\dr}^\eps = - \rho_1 \mathcal{J}_\eps ( \v \nabla \cdot \dot{\dr}^\eps ) + \Delta  \dr^\eps + \mathcal{J}_\eps \big{(} \gamma ( \v, \dr^\eps,  \dot{\dr}^\eps )  \dr^\eps  \big{)} \\
        \qquad\qquad\qquad + \lambda_1 \dot{\dr}^\eps - \lambda_1 \mathcal{J}_\eps ( \B \dr^\eps ) + \lambda_2 \mathcal{J}_\eps ( \A  \dr^\eps ) \,, \\
        \partial_t \dr^\eps = \dot{\dr}^\eps - \mathcal{J}_\eps ( \v \cdot \nabla \dr^\eps )\,, \\
        ( \dr^\eps, \dot{\dr}^\eps ) |_{t=0} = ( \mathcal{J}_\eps \dr_0, \mathcal{J}_\eps \tilde{\dr}_0 )\,.
      \end{array}
    \right.
  \end{equation}

  {\em Step 2. Uniform energy estimate.} We employ the standard energy estimate arguments to derive the uniform energy bounds, where the fact \eqref{Mollify-Indent}, the H\"older inequality and Sobolev embedding theory are frequently used.

  First, we calculate the $L^2$-estimate of the approximate system \eqref{Approx-Syst-WM-2}. Multiplying by $\dot{\dr}^\eps$ in the first equation of \eqref{Approx-Syst-WM-2} and integrating by parts on $\R^n$, we have
  \begin{equation}\label{L2-1}
    \begin{aligned}
      \tfrac{1}{2} \tfrac{\d}{\d t} \Big{(} \rho_1 |\dot{\dr}^\eps|^2_{L^2} & + |\nabla^{k+1} \dr^\eps|^2_{L^2} \Big{)} =  - \l \rho_1 \v \cdot \nabla \dot{\dr}^\eps, \dot{\dr}^\eps \r + \l \Delta \dr^\eps, \v \cdot \nabla \dr^\eps \r + \lambda_1 |\dot{\dr}^\eps|^2_{L^2} \\
      & +  \l \gamma (\v, \dr^\eps, \dot{\dr}^\eps) \dr^\eps, \dot{\dr}^\eps \r - \lambda_1 \l \B \dr^\eps, \dot{\dr}^\eps \r + \lambda_2 \l \A \dr^\eps, \dot{\dr}^\eps \r \,.
    \end{aligned}
  \end{equation}
  It is derived from the H\"older inequality and integration by parts on $\R^n$ that
  \begin{equation*}
    \begin{aligned}
      - \l \rho_1 \v \cdot \nabla \dot{\dr}^\eps, \dot{\dr}^\eps \r + \l \Delta \dr^\eps, \v \cdot \nabla \dr^\eps \r = & \l \tfrac{1}{2} \div \v, \rho_1 |\dot{\dr}^\eps|^2 - |\nabla \dr^\eps|^2 \r + \l \nabla \dr^\eps, \nabla \v \cdot \nabla \dr^\eps \r \\
      \lesssim & |\nabla \v|_{L^\infty} \big{(} \rho_1 |\dot{\dr}^\eps|^2_{L^2} + |\nabla \dr^\eps|^2_{L^2} \big{)} \,,
    \end{aligned}
  \end{equation*}
  and
  \begin{equation*}
    \begin{aligned}
      \l \gamma (\v, \dr^\eps, \dot{\dr}^\eps) \dr^\eps, \dot{\dr}^\eps \r = & \l ( - \rho_1 |\dot{\dr}^\eps|^2 + |\nabla \dr^\eps|^2 - \lambda_2 \dr^\eps{}\top \A \dr^\eps ) \dr^\eps, \dot{\dr}^\eps \r \\
      \lesssim & |\dr^\eps|_{L^\infty} |\dot{\dr}^\eps|_{L^\infty} \big{(} \rho_1 |\dot{\dr}^\eps|^2_{L^2} + |\nabla \dr^\eps|^2_{L^2} \big{)} + |\lambda_2| |\nabla \v|_{L^2} |\dr^\eps|^3_{L^\infty} |\dot{\dr}^\eps|_{L^2} \,,
    \end{aligned}
  \end{equation*}
  and
  \begin{equation*}
    \begin{aligned}
      - \lambda_1 \l \B \dr^\eps, \dot{\dr}^\eps \r + \lambda_2 \l \A \dr^\eps, \dot{\dr}^\eps \r \lesssim ( |\lambda_1| + |\lambda_2| ) |\nabla \v|_{L^2} |\dr^\eps|_{L^\infty} |\dot{\dr}^\eps|_{L^2}\,.
    \end{aligned}
  \end{equation*}
  Then, by plugging the above three inequalities into the equality \eqref{L2-1} we have the $L^2$-estimate
  \begin{equation}\label{L2}
    \begin{aligned}
      & \tfrac{1}{2} \tfrac{\d}{\d t} \Big{(} \rho_1 |\dot{\dr}^\eps|^2_{L^2} + |\nabla \dr^\eps|^2_{L^2} \Big{)} \lesssim  ( |\nabla \v|_{L^\infty} + |\lambda_1| ) \big{(} \rho_1 |\dot{\dr}^\eps|^2_{L^2} + |\nabla \dr^\eps|^2_{L^2} \big{)}  + |\lambda_2| |\nabla \v|_{L^2} |\dr^\eps|^3_{L^\infty} |\dot{\dr}^\eps|_{L^2} \\
      &+  |\dr^\eps|_{L^\infty} |\dot{\dr}^\eps|_{L^\infty} \big{(} \rho_1 |\dot{\dr}^\eps|^2_{L^2} + |\nabla \dr^\eps|^2_{L^2} \big{)} + ( |\lambda_1| + |\lambda_2| ) |\nabla \v|_{L^2} |\dr^\eps|_{L^\infty} |\dot{\dr}^\eps|_{L^2} \,.
    \end{aligned}
  \end{equation}

  Second, we estimate the higher order energy bounds of the approximate system \eqref{Approx-Syst-WM-2}. For all $1 \leq k \leq s$, we act the $k$-order derivative operator $\nabla^k$ on the first equation of the system \eqref{Approx-Syst-WM-2} and take $L^2$-inner product by multiplying $\nabla^k \dot{\dr}^\eps$, then by integrating by parts we obtain
  \begin{equation}\label{Hk-1}
    \begin{aligned}
      \tfrac{1}{2} \tfrac{\d}{\d t} & \Big{(} \rho_1 |\nabla^k \dot{\dr}^\eps|^2_{L^2} + |\nabla \dr^\eps|^2_{L^2} \Big{)} = \lambda_1 |\nabla^k \dot{\dr}^\eps|^2_{L^2} - \rho_1 \l \nabla^k (\v \cdot \nabla \dot{\dr}^\eps), \nabla^k \dot{\dr}^\eps \r  \\
      &  + \l \nabla^k ( \gamma (\v, \dr^\eps, \dot{\dr}^\eps) \dr^\eps ) , \nabla^k \dot{\dr}^\eps \r + \l \Delta \nabla^k \dr^\eps, \nabla^k ( \v \cdot \nabla \dr^\eps ) \r \\
      &- \lambda_1 \l \nabla^k (\B \dr^\eps), \nabla^k \dot{\dr}^\eps \r + \lambda_2 \l \nabla^k (\A \dr^\eps), \nabla^k \dot{\dr}^\eps \r \,.
    \end{aligned}
  \end{equation}
  Now we estimate the right hand side of the equality \eqref{Hk-1} term by term. It is implied by integrating by parts, the H\"older inequality and Sobolev embedding theory that
  \begin{equation}\label{Hk-2}
    \begin{aligned}
      & - \rho_1 \l \nabla^k (\v \cdot \nabla \dot{\dr}^\eps), \nabla^k \dot{\dr}^\eps \r =  - \rho_1 \sum_{\substack{a+b=k \\ a \geq 2}} \l \nabla^a \v \nabla^{b+1} \dot{\dr}^\eps , \nabla^k \dot{\dr}^\eps \r  \\
      & - \rho_1 \l \v \cdot \nabla \nabla^k \dot{\dr}^\eps , \nabla^k \dot{\dr}^\eps \r  - \rho_1 \l \nabla \v \nabla^k \dot{\dr}^\eps, \nabla^k \dot{\dr}^\eps \r \\
      \lesssim & \rho_1 |\nabla \v|_{L^\infty} |\nabla^k \dot{\dr}^\eps|^2_{L^2} + \rho_1 \sum_{\substack{a+b=k \\ a \geq 2}} |\nabla^a \v|_{L^4} |\nabla^{b+1} \dot{\dr}^\eps|_{L^4} |\nabla^k \dot{\dr}^\eps|_{L^2} \lesssim  \rho_1 |\nabla \v|_{H^s} |\dot{\dr}^\eps|^2_{H^s} \,,
    \end{aligned}
  \end{equation}
  where the first inequality is derived from the fact $ \langle \v \cdot \nabla \nabla^k \dot{\dr}^\eps , \nabla^k \dot{\dr}^\eps \rangle = \langle \frac{1}{2} \div \v, |\nabla^k \dot{\dr}^\eps|^2 \rangle $, and similar calculation in \eqref{Hk-2} reduces to
  \begin{equation}\label{Hk-3}
    \begin{aligned}
      \l \Delta \nabla^k \dr^\eps, \nabla^k ( \v \cdot \nabla \dr^\eps ) \r =  - \l \nabla^{k+1} (\v \cdot \nabla \dr^\eps), \nabla^{k+1} \dr^\eps \r \lesssim  |\nabla \v|_{H^s} |\nabla^{k+1} \dr^\eps|^2_{H^s} \,,
    \end{aligned}
  \end{equation}
  and
  \begin{equation}\label{Hk-4}
    \begin{aligned}
       & - \lambda_1 \l \nabla^k (\B \dr^\eps), \nabla^k \dot{\dr}^\eps \r  + \lambda_2 \l \nabla^k (\A \dr^\eps), \nabla^k \dot{\dr}^\eps \r \\
      \lesssim & ( |\lambda_1| + |\lambda_2| ) \Big( |\nabla^{k+1} \v|_{L^2} |\dr^\eps|_{L^\infty}   +  \sum_{ \substack{a+b=k \\ b \geq 1}} |\nabla^{a+1} \v|_{L^4} |\nabla^b \dr^\eps|_{L^4} \Big) |\nabla^k \dot{\dr}^\eps|_{L^2} \\
      \lesssim & ( |\lambda_1| + |\lambda_2| ) |\nabla \v|_{H^s} \big{(} |\dr^\eps|_{L^\infty} + |\nabla \dr^\eps|_{H^s} \big{)} |\dot{\dr}^\eps|_{H^s}\,.
    \end{aligned}
  \end{equation}

  It remains to estimate the term $\l \nabla^k ( \gamma (\v, \dr^\eps, \dot{\dr}^\eps) \dr^\eps ) , \nabla^k \dot{\dr}^\eps \r$, which can be divided into three parts. We can directly calculate the following term
    \begin{align}\label{Hk-5}
      \no & - \rho_1 \l \nabla^k (|\dot{\dr}^\eps|^2 \dr^\eps), \nabla \dot{\dr}^\eps \r \\
      \no \lesssim & \rho_1 |\dr^\eps|_{L^\infty} \sum_{a+b=k} \l |\nabla^a \dot{\dr}^\eps|\, |\nabla^b \dot{\dr}^\eps|, |\nabla^k \dot{\dr}^\eps| \r + \rho_1 \sum_{\substack{a+b+c=k \\ c \geq 1}} \l |\nabla^a \dot{\dr}^\eps|\, |\nabla^b \dot{\dr}^\eps|\, |\nabla^c \dr^\eps|, |\nabla^k \dot{\dr}^\eps| \r \\
     \lesssim & 2 \rho_1 |\dr^\eps|_{L^\infty} |\dot{\dr}^\eps|_{L^\infty} |\nabla^k \dot{\dr}^\eps|^2_{L^2} + \rho_1 |\dr^\eps|_{L^\infty} \sum_{\substack{a+b=k \\ a,b \geq 1}} |\nabla^a \dot{\dr}^\eps|_{L^4} |\nabla^b \dot{\dr}^\eps|_{L^4} |\nabla^k \dot{\dr}^\eps|_{L^2} \\
      \no & + \rho_1 |\dot{\dr}^\eps|_{L^\infty} |\dot{\dr}^\eps|_{L^4} |\nabla^k \dr^\eps|_{L^4} |\nabla^k \dot{\dr}^\eps|_{L^2} + \rho_1 \sum_{\substack{a+b+c=k \\ 1 \leq c \leq k-1}} |\nabla^c \dr^\eps|_{L^\infty} |\nabla^a \dot{\dr}^\eps|_{L^4} |\nabla^b \dot{\dr}^\eps|_{L^4} |\nabla^k \dot{\dr}^\eps|_{L^2} \\
      \no \lesssim & \rho_1 \big{(} |\dr^\eps|_{L^\infty} + |\nabla \dr^\eps|_{H^s} \big{)} |\dot{\dr}^\eps|^3_{H^s}\,,
    \end{align}
  where the last two inequalities are implied by the Sobolev embedding theory and H\"older inequality. According to the same calculation in \eqref{Hk-5}, we can estimate
  \begin{equation}\label{Hk-6}
    \l \nabla^k ( |\nabla \dr^\eps|^2 \dr^\eps ) , \nabla^k \dot{\dr}^\eps \r \lesssim  \big{(} |\dr^\eps|_{L^\infty} + |\nabla \dr^\eps|_{H^s} \big{)} |\nabla \dr^\eps|^2_{H^s} |\dot{\dr}^\eps|_{H^s} \,.
  \end{equation}
  Now we estimate the last term of $\l \nabla^k ( \gamma (\v, \dr^\eps, \dot{\dr}^\eps) \dr^\eps ) , \nabla^k \dot{\dr}^\eps \r$ by the H\"older inequality and Sobolev embedding. More precisely, we calculate
  \begin{equation}\label{Hk-7}
    \begin{aligned}
      & - \lambda_2  \l \nabla^k ( (\dr^\eps{}^\top \A \dr^\eps) \dr^\eps ), \nabla^k \dot{\dr}^\eps \r \\
      \lesssim & |\lambda_2| \sum_{a+b+c+e=k} \l |\nabla^a \dr^\eps|\, |\nabla^b \dr^\eps|\, |\nabla^c \dr^\eps|\, |\nabla^{e+1} \v| , |\nabla^k \dot{\dr}^\eps| \r \\
      \lesssim & |\lambda_2| |\dr^\eps|^3_{L^\infty} |\nabla^{k+1} \v|_{L^2} |\nabla^k \dot{\dr}^\eps|_{L^2} + 3 |\lambda_2| |\dr^\eps|^2_{L^\infty} \sum_{\substack{a+e=k \\ a \geq 1}} \l |\nabla^a \dr^\eps|\, |\nabla^{e+1} \v|, |\nabla^k \dot{\dr}^\eps| \r \\
      & + 3 |\lambda_2| |\dr^\eps|_{L^\infty} \sum_{\substack{a+b+e=k \\ a,b \geq 1}} \l |\nabla^a \dr^\eps|\, |\nabla^b \dr^\eps|\, |\nabla^{e+1} \v|, |\nabla^k \dot{\dr}^\eps| \r \\
      & + |\lambda_2|  \sum_{\substack{a+b+c+e=k \\ a,b,c \geq 1}} \l |\nabla^a \dr^\eps|\, |\nabla^b \dr^\eps|\, |\nabla^c \dr^\eps|\, |\nabla^{e+1} \v|, |\nabla^k \dot{\dr}^\eps| \r \\
      \lesssim & |\lambda_2| |\nabla \v |_{H^s} \big{(} |\dr^\eps|_{L^\infty} + |\nabla \dr^\eps|_{H^s} \big{)}^3 |\dot{\dr}^\eps|_{H^s}\,.
    \end{aligned}
  \end{equation}
  Therefore, by substituting the inequalities \eqref{Hk-2}, \eqref{Hk-3}, \eqref{Hk-4}, \eqref{Hk-5}, \eqref{Hk-6} and \eqref{Hk-7} into the equality \eqref{Hk-1}, we know that the inequality
  \begin{equation}\label{Hk}
    \begin{aligned}
      & \tfrac{1}{2} \tfrac{\d}{\d t} \Big{(} \rho_1 |\nabla^k \dot{\dr}^\eps|^2_{L^2} + |\nabla^{k+1} \dr^\eps|^2_{L^2} \Big{)} \lesssim |\lambda_1| | \nabla^k \dot{\dr}^\eps |^2_{L^2} + |\nabla \v|_{H^s} \big{(} \rho_1 |\dot{\dr}^\eps|^2_{H^s} + |\nabla \dr^\eps|^2_{H^s} \big{)} \\
      & + ( |\lambda_1| + |\lambda_2| ) |\nabla \v|_{H^s} \big{(} |\dr^\eps|_{L^\infty} + |\nabla \dr^\eps|_{H^s} \big{)} |\dot{\dr}^\eps|_{H^s} \\
      & + \big{(} |\dr^\eps|_{L^\infty} + |\nabla \dr^\eps|_{H^s} \big{)} \big{(} \rho_1 |\dot{\dr}^\eps|^2_{H^s} + |\nabla \dr^\eps|^2_{H^s} \big{)} |\dot{\dr}^\eps|_{H^s}  + |\lambda_2| |\nabla \v|_{H^s} \big{(} |\dr^\eps|_{L^\infty} + |\nabla \dr^\eps|_{H^s} \big{)}^3 |\dot{\dr}^\eps|_{H^s}
    \end{aligned}
  \end{equation}
  holds for all $1 \leq k \leq s$. Combining the $L^2$-estimate \eqref{L2} and higher order derivative estimate \eqref{Hk}, we have
  \begin{equation}\label{Hs}
    \begin{aligned}
      & \tfrac{1}{2} \tfrac{\d}{\d t} \Big{(} \rho_1 |\dot{\dr}^\eps|^2_{H^s} + |\nabla \dr^\eps|^2_{H^s} \Big{)} \lesssim |\lambda_1| | \nabla^k \dot{\dr}^\eps |^2_{L^2} + |\nabla \v|_{H^s} \big{(} \rho_1 |\dot{\dr}^\eps|^2_{H^s} + |\nabla \dr^\eps|^2_{H^s} \big{)} \\
      & + ( |\lambda_1| + |\lambda_2| ) |\nabla \v|_{H^s} \big{(} |\dr^\eps|_{L^\infty} + |\nabla \dr^\eps|_{H^s} \big{)} |\dot{\dr}^\eps|_{H^s} + |\lambda_2| |\nabla \v|_{H^s} \big{(} |\dr^\eps|_{L^\infty} + |\nabla \dr^\eps|_{H^s} \big{)}^3 |\dot{\dr}^\eps|_{H^s}\ \\
      & + \big{(} |\dr^\eps|_{L^\infty} + |\nabla \dr^\eps|_{H^s} \big{)} \big{(} \rho_1 |\dot{\dr}^\eps|^2_{H^s} + |\nabla \dr^\eps|^2_{H^s} \big{)} |\dot{\dr}^\eps|_{H^s} \,.
    \end{aligned}
  \end{equation}

  Third, we notice that the norm $|\dr^\eps|_{L^\infty}$ in the above $H^s$-energy estimate \eqref{Hs} is uncontrolled by now, so that we should try to control it. We observe that
  \begin{equation}\label{L-infty}
    \begin{aligned}
      |\dr^\eps|_{L^\infty} \lesssim & |\dr^\eps - \mathcal{J}_\eps \dr_0 |_{L^\infty} + |\mathcal{J}_\eps \dr_0|_{L^\infty} \lesssim  |\dr^\eps - \mathcal{J}_\eps \dr_0 |_{H^2} + 1 \\
      \lesssim & |\dr^\eps - \mathcal{J}_\eps \dr_0 |_{L^2} + |\nabla \dr^\eps|_{H^1} + |\nabla \dr_0|_{H^1} + 1 \,,
    \end{aligned}
  \end{equation}
  where the second inequality is derived from the Sobolev embedding theory. Thanks to the relation $\partial_t \dr^\eps = \dot{\dr}^\eps - \mathcal{J}_\eps (\v \cdot \nabla \dr^\eps)$ and the Sobolev embedding theory, we can calculate
  \begin{equation}\label{L2-Diffe}
    \begin{aligned}
      \tfrac{1}{2} \tfrac{\d}{\d t} |\dr^\eps - \mathcal{J}_\eps \dr_0|^2_{L^2} = & \l \dr^\eps - \mathcal{J}_\eps \dr_0, \dot{\dr}^\eps - \mathcal{J}_\eps (\v \cdot \nabla \dr^\eps) \r \\
      \lesssim & |\dr^\eps - \mathcal{J}_\eps \dr_0|_{L^2} |\dot{\dr}^\eps|_{L^2} + |\v|_{H^2} |\dr^\eps - \mathcal{J}_\eps \dr_0|_{L^2} |\nabla \dr^\eps|_{L^2} \,.
    \end{aligned}
  \end{equation}
  If we define the energy functional $E_\eps (t)$ as
  \begin{equation*}
    E_\eps (t) = \rho_1 |\dot{\dr}^\eps|^2_{H^s} + |\nabla \dr^\eps|^2_{H^s} + |\dr^\eps - \mathcal{J}_\eps \dr_0|^2_{L^2}\,,
  \end{equation*}
  then the inequalities \eqref{Hs}, \eqref{L-infty} and \eqref{L2-Diffe} imply that there is a positive constant $C_1$, depending only on $\lambda_1$, $\lambda_2$, $\rho_1$ and $\dr_0$, such that for all $\eps > 0$ the inequality
  \begin{equation}\label{Energy-Est-1}
    \tfrac{\d}{\d t} E_\eps (t) \leq C_1 ( 1 + | \v|_{H^{s+1}} ) [ 1 + E_\eps (t) ]^2
  \end{equation}
  holds for all $t \in [ 0, T_\eps)$.

  In the end, we will evaluate the uniform bounds of the energy functional $E_\eps (t)$ by using Gronwall arguments. Noticing that
  \begin{equation*}
    E_\eps (0) = \rho_1 |\tilde{\dr}_0|^2_{H^s} + |\nabla \dr_0|^2_{H^s} \equiv E^{in} < \infty \, ,
  \end{equation*}
  we define $T_\eps^1$ as
  \begin{equation*}
    T_\eps^1 = \Big{\{} \tau \in [ 0, T_\eps ) ; \sup_{t \in [0,\tau]} E_\eps(t) \leq 2 E^{in} \Big{\}} \geq 0\,.
  \end{equation*}
  It is immediately derived from the continuity of the energy functional $E_\eps(t)$ that $T_\eps^1 > 0$. Then the inequality \eqref{Energy-Est-1} implies that for all $t \in [0,T_\eps^1]$
  \begin{equation*}
    \tfrac{\d}{\d t} E_\eps (t) \leq \Lambda(t)  [ 1 + E_\eps (t) ] \,,
  \end{equation*}
  where the non-negative function $\Lambda(t) = C_1 (1 + 2 E^{in}) ( 1 + | \v|_{H^{s+1}} ) \in L^1 (0,T_0) $. So Gronwall inequality reduces to
  $$E_\eps (t) \leq G (t) \equiv \bigg{(} E^{in} + \int_0^t \Lambda(\tau) \d \tau \bigg{)} \exp{\int_0^t \Lambda(\tau) \d \tau} $$
  holds for all $\eps > 0$, where $G(0) = E^{in} > 0$. Since the function $G(t)$ is continuous in $t$ and is independent of $\eps > 0$, there is a $T > 0$ independent of $\eps > 0$ such that $G(t) \leq 2 E^{in}$ for all $t \in [0,T]$. Hence $T_\eps^1 \geq T > 0$ for all $\eps > 0$. Therefore, for all $\eps > 0$ and $t \in [0,T]$, we have $E_\eps (t) \leq 2 E^{in}$. Namely, we obtain the following uniform energy bound
  \begin{equation}\label{Energy-Est-2}
    \rho_1 |\dot{\dr}^\eps|^2_{H^s} + |\nabla \dr^\eps|^2_{H^s} + |\dr^\eps - \mathcal{J}_\eps \dr_0|^2_{L^2} \leq 2 E^{in}
  \end{equation}
  for all $\eps > 0$ and $t \in [0,T]$.

  {\em Step 3. Pass to the limits.} By the bounds \eqref{L-infty} and \eqref{Energy-Est-2}, we know that there is a $\dr \in L^\infty([0,T] \times \R^n)$ satisfying $\nabla \dr \,,\ \dot{\dr} \in C(0,T;H^s)$ such that $\dr$ obey the first equation of \eqref{WM} with the initial conditions \eqref{IC-WM} after passing limits in the approximate system \eqref{Approx-Syst-WM-2} as $\eps \rightarrow 0$, and we know that
  \begin{equation*}
    \left\{
      \begin{array}{l}
        \rho_1 \ddot{\dr} = \Delta \dr + \gamma (\v, \dr, \dot{\dr}) \dr + \lambda_1 (\dot{\dr} + \B \dr) + \lambda_2 \A \dr \,, \\
        (\dr , \dot{\dr})|_{t=0} = ( \dr^{in}(x), \tilde{\dr}^{in} (x) ) \in \mathbb{S}^{n-1} \times \R^n
      \end{array}
    \right.
  \end{equation*}
  with $\dr \in L^\infty([0,T] \times \R^n)$, where $$ \gamma (\v, \dr, \dot{\dr}) = - \rho_1 |\dot{\dr}|^2 + |\nabla \dr|^2  - \lambda_2 \dr^\top \A \dr \,. $$
  Then Lemma \ref{Parallel-d-Lemma} tells us that $\dr \in \mathbb{S}^{n-1}$, and the proof of Proposition \ref{Prop-WellPosed-WM} is finished.

\end{proof}

\section{The iterating approximate system}\label{Sec-IAS}

In this section, we construct the approximate system by iteration. More precisely, the iterating approximate system is constructed as follows: for all integer $k \geq 0$
\begin{equation}\label{Iter-Appr-Syt}
  \left\{
    \begin{array}{c}
      \partial_t \v^{k+1} + \v^k \cdot \nabla \v^k - \tfrac{1}{2} \Delta \v^{k+1} + \nabla p^{k+1} = - \div ( \nabla \dr^k \odot \nabla \dr^k ) + \div \tilde{\sigma} ( \v^{k+1} , \dr^k , \dot{\dr}^k ) \,, \\
      \div \v^{k+1} = 0\,, \\
      \rho_1 \partial_t \dot{\dr}^{k+1} + \rho_1 \v^k \cdot \nabla \dot{\dr}^{k+1} = \Delta \dr^{k+1} + \gamma (\v^k , \dr^{k+1}, \dot{\dr}^{k+1}) \dr^{k+1} \\
      \qquad\qquad\qquad\qquad\qquad\qquad\quad + \lambda_1 ( \dot{\dr}^{k+1} + \B^k \dr^{k+1} ) + \lambda_2 \A^k \dr^{k+1} \,, \\
      ( \v^{k+1}, \dr^{k+1} , \dot{\dr}^{k+1} )\big{|}_{t=0} = ( \v^{in}(x), \dr^{in}(x), \tilde{\dr}^{in}(x) ) \in \R^n \times \mathbb{S}^{n-1} \times \R^n \,,
    \end{array}
  \right.
\end{equation}
where $\dot{\dr}^{k+1} = \partial_t \dr^{k+1} + \v^k \cdot \nabla \dr^{k+1}$ is the iterating approximate material derivatives, and
$$\A^k = \tfrac{1}{2} ( \nabla \v^k + \nabla \v^k{}^\top ) \,, \B^k = \tfrac{1}{2} ( \nabla \v^k - \nabla \v^k{}^\top ) \,,$$
the iterating approximate Lagrangian multiplier $\gamma (\v^k , \dr^{k+1}, \dot{\dr}^{k+1}) $ is
\begin{equation*}
  \gamma (\v^k, \dr^{k+1}, \dot{\dr}^{k+1}) = - \rho_1 |\dot{\dr}^{k+1}|^2 + |\nabla \dr^{k+1}|^2 - \lambda_2 \dr^{k+1}{}^\top \A^k \dr^{k+1} \,,
\end{equation*}
and
\begin{equation*}
  \begin{aligned}
    \big{(}\tilde{\sigma}( \v^{k+1} , \dr^k , \dot{\dr}^k ) \big{)}_{ji}  = & \mu_1 \dr^k_q \dr^k_p \A^{k+1}_{qp} \dr^k_i \dr^k_j + \mu_2 \dr^k_j (\dot{\dr}^k_i + \B^{k+1}_{qi} \dr^k_q)  \\
    &+ \mu_3 \dr^k_i (\dot{\dr}^k_j + \B^{k+1}_{qj} \dr^k_q)    +  \mu_5 \dr^k_j \dr^k_q \A^{k+1}_{qi} + \mu_6 \dr^k_i \dr^k_q \A^{k+1}_{qj}
  \end{aligned}
\end{equation*}
is the $(j,i)$-entry of the iterating approximate extra stress tensor $\tilde{\sigma}( \v^{k+1} , \dr^k , \dot{\dr}^k )$. The iteration starts from $k=0$, i.e.
\begin{equation*}
  ( \v^0 (t,x), \dr^0(t,x) , \dot{\dr}^0 (t,x) ) = ( \v^{in}(x), \dr^{in}(x), \tilde{\dr}^{in}(x) )\,.
\end{equation*}

Now we state the existence conclusion of the iterating approximate system \eqref{Iter-Appr-Syt} as follows:
\begin{lemma}\label{Lm-Ext-Iter-Appr-Syst}
  Suppose that $s > \frac{n}{2} + 1$ and the initial data $( \v^{in}, \dr^{in}, \tilde{\dr}^{in} ) \in \R^n \times \mathbb{S}^{n-1} \times \R^n$ satisfy $ \v^{in}\,, \nabla \dr^{in}\,, \tilde{\dr}^{in} \in H^s $. Then there is a maximal number $T^*_{k+1} > 0$ such that the system \eqref{Iter-Appr-Syt} admits a unique solution $ ( \v^{k+1}, \dot{\dr}^{k+1}, \dr^{k+1} ) $ satisfying $\v^{k+1} \in C(0, T_{k+1}^*;H^s) \cap L^2(0, T_{k+1}^*;H^{s+1})$, and $\nabla \dr^{k+1} \,, \dot{\dr}^{k+1} \in C(0, T_{k+1}^*;H^s) $.
\end{lemma}

\begin{proof}

  For the case $k+1$, the vectors $\v^k$, $\dr^k$ and $ \dot{\dr}^k$ are the known. Namely, the velocity equation of $\v^{k+1}$ is a linear Stokes type system, which admits the unique solution $\v^{k+1} \in C(0,\hat{T}_{k+1}; H^s) \cap L^2 (0,\hat{T}_{k+1}; H^{s+1})$ on the maximal time interval $[0,\hat{T}_{k+1} )$; and the orientation equation of $\dr^{k+1}$ is a Ericksen-Leslie's hyperbolic liquid crystal flow with a given bulk velocity $\v^k$, which, by Proposition \ref{Prop-WellPosed-WM}, has the unique solution $\dr^{k+1}$ satisfying $\nabla \dr^{k+1} \,, \ \dot{\dr}^{k+1} \in C (0,\tilde{T}_{k+1}; H^s)$ on the maximal time interval $[0,\tilde{T}_{k+1} )$. We denote $T^*_{k+1} = \min \{ \hat{T}_{k+1}, \tilde{T}_{k+1} \} > 0 $, and the proof of Lemma \ref{Lm-Ext-Iter-Appr-Syst} is finished.

\end{proof}

We remark that $T^*_{k+1} \leq T^*_{k}$.

\section{Local well-posedness with large initial data}\label{Sec-LWP}

In this section, we prove the local well-posedness of the Ericksen-Leslie's hyperbolic liquid crystal flow \eqref{PHLC}-\eqref{Inital-Data} with large initial data under the Leslie coefficients constraints $\mu_1 \geq 0\,, \mu_4 > 0\,, \lambda_1 < 0\,, \mu_5 + \mu_6 + \tfrac{\lambda_2^2}{\lambda_1} \geq 0$ or $\mu_1 \geq 0\,, \mu_4 > 0\,, \lambda_1 = 0\,, (1 - \delta ) \mu_4 ( \mu_5 + \mu_6 ) \geq 2 | \lambda_2 |^2$ for some $\delta \in (0,1)$, i.e., the first part of Theorem \ref{Main-Thm}. The key point is to justify the positive lower bound of $T^*_{k+1}$ and the uniform energy bounds of the iterating approximate system \eqref{Iter-Appr-Syt}, which will be shown in Lemma \ref{Lm-Lower-Bnds}. In the end, by the compactness arguments and Lemma \ref{Parallel-d-Lemma}, we can pass to the limits in the system \eqref{Iter-Appr-Syt} and then reach our goal, which is a standard proceeding. We define the following energy functionals:
\begin{equation*}
    \begin{aligned}
      E_{k+1} (t) &= |\v^{k+1}|^2_{H^s} + \rho_1 |\dot{\dr}^{k+1}|^2_{H^s} + |\nabla \dr^{k+1}|^2_{H^s} \,, \\
      D_{k+1} (t) &= \tfrac{1}{2} \mu_4 |\nabla \v^{k+1}|^2_{H^s}\,,
    \end{aligned}
  \end{equation*}
  and precisely state our key lemma:

\begin{lemma}\label{Lm-Lower-Bnds}
  Assume that $(\v^{k+1}, \dr^{k+1})$ is the solution to the iterating approximate system \eqref{Iter-Appr-Syt} and we define
   \begin{equation*}
     T_{k+1} \equiv \sup \Big{\{} \tau \in [0, T^*_{k+1} ) ; \sup_{t \in [0,\tau]}  E_{k+1} (t) + \int_0^\tau  D_{k+1} (t) \d t \leq M \Big{\}} \,,
   \end{equation*}
   where $T^*_{k+1} > 0$ is the existence time of the iterating approximate system \eqref{Iter-Appr-Syt}. Then for any fixed $M > E^{in}$ there is a constant $T > 0$, depending only on Leslie coefficients, $M$ and $E^{in}$, such that
   $$ T_{k+1} \geq T > 0\,. $$
\end{lemma}

\begin{proof}
  By the continuity of the functionals $E_{k+1} (t)$, we know that $T_{k+1} > 0$. If the sequence $\{T_{k}; k =1,2,\cdots \}$ is increasing, the conclusion immediately holds. So we consider that the sequence $T_k$ is not increasing. Now we choose a strictly increasing sequence $\{ k_p \}_{p=1}^\Lambda$ as follows:
  \begin{equation*}
    k_1 = 1\,, k_{p+1} = \min \big{\{} k;\, k > k_p\,, T_k < T_{k_p} \big{\}}\,.
  \end{equation*}
  If $\Lambda < \infty$, the conclusion holds. Consequently, we consider the case $\Lambda = \infty$. By the definition of $k_p$, the sequence $T_{k_p}$ is strictly increasing, so that our goal is to prove
  $$ \lim_{p \rightarrow \infty} T_{k_p} > 0 \,. $$

  We first estimate the energy of the iterating approximate system \eqref{Iter-Appr-Syt}. For all $0 \leq l \leq s$, we take $\nabla^l$ in the third equation of \eqref{Iter-Appr-Syt} and multiply $\nabla^l \dot{\dr}^{k+1}$ by integration on $\R^n$. Then by the similar estimate arguments in a priori estimates in Lemma \ref{Lm-APE}, we gain
  \begin{equation}\label{Iter-Est-1}
    \begin{aligned}
      & \tfrac{1}{2} \tfrac{\d}{\d t} \Big{(} \rho_1 |\dot{\dr}^{k+1}|^2_{H^s} + |\nabla \dr^{k+1}|^2_{H^s} \Big{)} \lesssim  \big{(} \rho_1 |\dot{\dr}^{k+1}|^2_{H^s} + |\nabla \dr^{k+1}|^2_{H^s} \big{)} |\nabla \v^k|_{H^s} \\
      & + \rho_1 |\dot{\dr}^{k+1}|^3_{H^s} |\nabla \dr^{k+1}|_{H^s} + |\dot{\dr}^{k+1}|_{H^s} |\nabla \dr^{k+1}|^3_{H^s} \\
      & + ( |\lambda_1| + |\lambda_2| ) ( 1 + |\nabla \dr^{k+1}|_{H^s}  ) |\dot{\dr}^{k+1}|_{H^s} |\nabla \v^k|_{H^s}  + |\lambda_1| |\nabla \v^k|_{H^s} |\dot{\dr}^{k+1}|_{H^s} \sum_{p=1}^3 |\nabla \dr^{k+1}|^p_{H^s} \,.
    \end{aligned}
  \end{equation}

  For all $0 \leq l \leq s$, we act $\nabla^l$ on the first equation of \eqref{Iter-Appr-Syt} and take $L^2$-inner product with $\nabla^l \v^{k+1}$. Then it is similarly derived from a priori estimate in Lemma \ref{Lm-APE} that for the case $\lambda_1 < 0$, i.e. \eqref{Coeffs-3}
    \begin{align}\label{Iter-Est-2}
      \no & \tfrac{1}{2} \tfrac{\d}{\d t} |\v^{k+1}|^2_{H^s} + \tfrac{1}{2} \mu_4 |\nabla \v^{k+1}|^2_{H^s} + \mu_1 \sum_{l=0}^s | \dr^k{}^\top ( \nabla^l \A^{k+1} ) \dr^k |^2_{L^2} \\
      \no & - \lambda_1 \sum_{l=0}^s |(\nabla^l \B^{k+1}) \dr^k + \tfrac{\lambda_2}{\lambda_1} (\nabla^l \A^{k+1}) \dr^k |^2_{L^2} + ( \mu_5 + \mu_6 + \tfrac{\lambda_2^2}{\lambda_1} ) \sum_{l=0}^s | (\nabla^l \A^{k+1}) \dr^k |^2_{L^2} \\
      \lesssim & ( |\v^k|^2_{H^s} + |\nabla \dr^k|^2_{H^s} ) |\nabla \v^{k+1}|_{H^s} + ( |\mu_2| + |\mu_3| ) ( 1 + |\nabla \dr^k|_{H^s} ) |\dot{\dr}^k|_{H^s} |\nabla \v^{k+1}|_{H^s} \\
      \no & + ( \mu_1 +  |\mu_2| + |\mu_3| + |\mu_5| + |\mu_6| ) \sum_{p=0}^4 |\nabla \dr^k|^p_{H^s} |\v^{k+1}|_{H^s} |\nabla \v^{k+1}|_{H^s} \,.
    \end{align}

  Then the inequalities \eqref{Iter-Est-1} and \eqref{Iter-Est-2} imply that
  \begin{equation}\label{Iter-Est-2*}
    \begin{aligned}
      & \tfrac{1}{2} \tfrac{\d}{\d t} \Big{(} |\v^{k+1}|^2_{H^s} + \rho_1 |\dot{\dr}^{k+1}|^2_{H^s} + |\nabla \dr^{k+1}|^2_{H^s} \Big{)}  + \mu_1 \sum_{l=0}^s | \dr^k{}^\top ( \nabla^l \A^{k+1} ) \dr^k |^2_{L^2} + \tfrac{1}{2} \mu_4 |\nabla \v^{k+1}|^2_{H^s} \\
      & - \lambda_1 \sum_{l=0}^s |(\nabla^l \B^{k+1}) \dr^k + \tfrac{\lambda_2}{\lambda_1} (\nabla^l \A^{k+1}) \dr^k |^2_{L^2} + ( \mu_5 + \mu_6 + \tfrac{\lambda_2^2}{\lambda_1} ) \sum_{l=0}^s | (\nabla^l \A^{k+1}) \dr^k |^2_{L^2} \\
      \lesssim & ( |\v^k|^2_{H^s} + |\nabla \dr^k|^2_{H^s} ) |\nabla \v^{k+1}|_{H^s} + ( |\mu_2| + |\mu_3| ) ( 1 + |\nabla \dr^k|_{H^s} ) |\dot{\dr}^k|_{H^s} |\nabla \v^{k+1}|_{H^s} \\
      & + ( \mu_1 +  |\mu_2| + |\mu_3| + |\mu_5| + |\mu_6| ) \sum_{p=0}^4 |\nabla \dr^k|^p_{H^s} |\v^{k+1}|_{H^s} |\nabla \v^{k+1}|_{H^s} \\
      & + \big{(} \rho_1 |\dot{\dr}^{k+1}|^2_{H^s} + |\nabla \dr^{k+1}|^2_{H^s} \big{)} |\nabla \v^k|_{H^s}  + \rho_1 |\dot{\dr}^{k+1}|^3_{H^s} |\nabla \dr^{k+1}|_{H^s} + |\dot{\dr}^{k+1}|_{H^s} |\nabla \dr^{k+1}|^3_{H^s} \\
      & + ( |\lambda_1| + |\lambda_2| ) ( 1 + |\nabla \dr^{k+1}|_{H^s}  ) |\dot{\dr}^{k+1}|_{H^s} |\nabla \v^k|_{H^s}  + |\lambda_1| |\nabla \v^k|_{H^s} |\dot{\dr}^{k+1}|_{H^s} \sum_{p=1}^3 |\nabla \dr^{k+1}|^p_{H^s} \,,
    \end{aligned}
  \end{equation}
  which immediately reduce to
  \begin{equation}\label{Iter-Est-3}
      \tfrac{\d}{\d t} E_{k+1} (t) + D_{k+1} (t) \leq C_1 ( 1 + E_k^8 (t) + D_k^\frac{1}{2} (t) ) \big{[} 1  +  E_{k+1} (t) \big{]}^2
  \end{equation}
  for all $t \in [ 0, T^*_{k+1} )$, where the constant $C_1 > 0$ depends only on the Leslie coefficients and the inertia density constant $\rho_1$. Here we make use of the definition of the functionals $E_{k} (t)$, $D_k (t)$ and the coefficients conditions \eqref{Coeffs-3}, i.e. $ \mu_1 \geq 0\,, \mu_4 > 0\,, \lambda_1 < 0\,, \mu_5 + \mu_6 + \tfrac{\lambda_2^2}{\lambda_1} \geq 0$. Actually, for the case $\lambda_1 = 0$, i.e. \eqref{Coeffs-4}, we can gain an inequality replacing the last three terms in the left-hand side of the inequality \eqref{Iter-Est-2*} with the terms
  \begin{equation*}
    \begin{aligned}
      & \tfrac{1}{2} \delta \mu_4 |\nabla \v^{k+1} |^2_{H^s} + ( \mu_5 + \mu_6 - \tfrac{2 \lambda_2^2}{ (1-\delta) \mu_4 } ) \sum_{l=0}^s |( \nabla^l \A^{k+1} ) \dr^k|^2_{L^2} \\
      & + \tfrac{1}{2} ( 1 - \delta ) \mu_4 \sum_{l=0}^s \big( |\nabla^{l+1} \v^{k+1}|_{L^2} - \tfrac{2 |\lambda_2| }{ (1 - \delta) \mu_4 } | (\nabla^l \A^{k+1}) \dr^k |_{L^2} \big)^2 \,.
    \end{aligned}
  \end{equation*}
Thus, we also gain the inequality \eqref{Iter-Est-3} by making use of the conditions \eqref{Coeffs-4}.

Recalling the definition of the sequence $\{ k_p \}$, we know that for any integer $N < k_p$,
$$ T_N > T_{k_p}\,. $$
We take $k= k_p - 1$ in the inequality \eqref{Iter-Est-3}, and then by the definition of $T_k$ we have for all $t \in [ 0, T_{k_p} ]$
\begin{equation}\label{Iter-Est-4}
  \tfrac{\d}{\d t} E_{k_p} (t) + D_{k_p} (t) \leq L_{k_p -1 }(t) \big{[} 1 + E_{k_p} (t) \big{]}^2\,,
\end{equation}
where the functions $ L_{k_p -1 }(t) = C_1 \big{[} 1 + M^8 +  D^\frac{1}{2}_{k_p -1 } (t) \big{]} \in L^2 ([ 0,T_{k_p }]) $. Noticing that $E_{k_p} (0) = E^{in}$, we solve the ODE inequality \eqref{Iter-Est-4} that for all $t \in [ 0, T_{k_p} ]$
\begin{equation*}
  E_{k_p} (t) \leq \frac{ E^{in} + C_1 ( 1 + E^{in} ) \big{[} (1+M^8) t + M^\frac{1}{2} t^\frac{1}{2} \big{]} }{ 1 - C_1 ( 1 + E^{in} ) \big{[} (1+M^8) t + M^\frac{1}{2} t^\frac{1}{2} \big{]}  } \equiv \mathcal{G}(t)\,,
\end{equation*}
where the function $\mathcal{G}(t)$ is strictly increasing and continuous and $\mathcal{G}(0) = E^{in}$. Plugging the above inequality into the ODE inequality \eqref{Iter-Est-4} and then integrating on $[0,t]$ for any $t \in [ 0, T_{k_p} ]$, we estimate that
\begin{equation}\label{Iter-Est-5}
  E_{k_p} + \int_0^t D_{k_p} (\tau) \d \tau \leq E^{in} + C_1 ( 1 + \mathcal{G} (t) )^2 \big{[} (1+M^8) t + M^\frac{1}{2} t^\frac{1}{2} \big{]} \equiv \mathcal{H}(t)\,,
\end{equation}
where the function $\mathcal{H}(t)$ is also strictly increasing and continuous and $\mathcal{H}(0) = E^{in}$. Thus, by the continuity and monotonicity of the function $ \mathcal{H}(t) $, we know that for any $M > E^{in}$ and $p \in \mathbb{N}^+$, there is a number $t^* > 0$, depending only on $M$, initial energy $E^{in}$, Leslie coefficients and inertia density constant $\rho_1$, such that for all $t \in [0, t^*]$
$$ E_{k_p} + \int_0^t D_{k_p} (\tau) \d \tau \leq M \,. $$
By the definition of $T_k$ we derive that $T_{k_p} \geq t^* > 0$, hence $T = \lim\limits_{p \rightarrow \infty} T_{k_p} \geq t^* > 0 $. Consequently, we complete the proof of Lemma \ref{Lm-Lower-Bnds}.

\end{proof}

\noindent{\bf Proof of the first part of Theorem \ref{Main-Thm}: Local well-posedness}. By Lemma \ref{Lm-Lower-Bnds} we know that for any fixed $M > E^{in}$ there is a $T > 0$ such that for all integer $k \geq 0$ and $t \in [0,T]$
\begin{equation}\label{Unif-Bnds}
  \sup_{t \in [0,T]} \Big{(} |\v^{k+1}|^2_{H^s} + \rho_1 |\dot{\dr}^{k+1}|^2_{H^s} + |\nabla \dr^{k+1}|^2_{H^s} \Big{)} + \tfrac{1}{2} \mu_4 \int_0^T |\nabla \v^{k+1}|^2_{H^s} \d t \leq M \,.
\end{equation}
Then, by compactness arguments and Lemma \ref{Parallel-d-Lemma}, we gain vectors $(\v, \dr) \in \R^n \times \mathbb{S}^{n-1}$ satisfying $\v \in L^\infty(0,T;H^s) \cap L^2 ( 0,T; H^{s+1} )$ and $\dot{\dr}\,, \nabla \dr \in L^\infty(0,T;H^s) $, which solve Ericksen-Leslie's hyperbolic liquid crystal flow \eqref{PHLC} with the initial conditions \eqref{Inital-Data}. Moreover, $(\v, \dr)$ admit the bound
\begin{equation*}
  \sup_{t \in [0,T]} \Big{(} |\v|^2_{H^s} + \rho_1 |\dot{\dr}|^2_{H^s} + |\nabla \dr|^2_{H^s} \Big{)} + \tfrac{1}{2} \mu_4 \int_0^T |\nabla \v|^2_{H^s} \d t \leq M \,.
\end{equation*}
Then the proof of the first part of Theorem \ref{Main-Thm} is finished.

\section{Global existence with small initial data}\label{Sec-Global}

In this section, we prove the global well-posedness of the system \eqref{PHLC} with small initial data in an {\em additional} coefficients constraint $\lambda_1 < 0$, namely, we verify the second part of Theorem \ref{Main-Thm}. We notice that the $H^s$-estimate \eqref{APE-Hs} does not have enough dissipation, if we aim at constructing the global solution with small initial data. Thus, we need to find a new energy estimate which involves enough dissipated terms. We first introduce the following energy functionals for any $\eta > 0$:
\begin{equation*}
  \begin{aligned}
    \mathcal{E}_\eta (t) =& |\v|^2_{H^s} + ( - \eta \lambda_1 + 1 - \eta \rho_1 ) |\nabla \dr|^2_{H^{s-1}} + |\nabla^{s+1} \dr|^2_{L^2} \\
    & + \rho_1 ( 1 - \eta ) |\dot{\dr}|^2_{{H}^s} + \rho_1 \eta |\dot{\dr}|^2_{L^2} + \eta \rho_1 |\dot{\dr} + \dr|^2_{\dot{H}^s} \,, \\
    \mathcal{D}_\eta (t) =& \tfrac{1}{4} \mu_4 |\nabla \v|^2_{H^s} + \tfrac{1}{4} \eta |\nabla \dr|^2_{\dot{H}^s} - \tfrac{1}{2} \lambda_1 \sum_{k=0}^s |\nabla^k \dot{\dr} + (\nabla^k \B) \dr + \tfrac{\lambda_2}{\lambda_1} (\nabla^k \A) \dr |^2_{L^2} \\
    & + \mu_1 \sum_{k=0}^s |\dr^\top (\nabla^k \A) \dr|^2_{L^2} + ( \mu_5 + \mu_6 + \tfrac{\lambda_2^2}{\lambda_1}  ) \sum_{k=0}^s |(\nabla^k \A) \dr|^2_{L^2} \\
    & + 3 \eta \rho_1 \sum_{k=1}^s | (\nabla^k \B) \dr + \tfrac{\lambda_2}{\lambda_1} (\nabla^k \A) \dr |^2_{L^2} \,.
  \end{aligned}
\end{equation*}
Then we establish the following key lemma to prove the global existence:

\begin{lemma}\label{Lm-Global}
  There exists a small $\eta_0 > 0$, depending only on Leslie coefficients and inertia density constant $\rho_1$, such that if $(\v,\dr)$ is the local solution constructed in the first part of Theorem \ref{Main-Thm}, then for all $0 < \eta \leq \eta_0$
  \begin{equation*}
    \begin{aligned}
      \tfrac{1}{2} \tfrac{\d}{\d t} \mathcal{E}_\eta (t) + \mathcal{D}_\eta (t) \leq C \sum_{p=1}^4 \mathcal{E}_\eta^\frac{p}{2} (t) \mathcal{D}_\eta (t) \,,
    \end{aligned}
  \end{equation*}
  where the positive constant $C$ depends only upon the inertia density constant $\rho_1$ and Leslie coefficients.
\end{lemma}

\begin{remark}
  The small $\eta_0 > 0$ is chosen to guarantee that the coefficients of the energy functional $ \mathcal{E}_\eta (t) $ are positive for all $\eta \in (0, \eta_0]$. The coefficients constraints $\mu_1 \geq 0\,, \mu_4 > 0\,, \lambda_1 < 0 \,, \mu_5 + \mu_6 + \tfrac{\lambda_2^2}{\lambda_1} \geq 0$ automatically imply the non-negativity of the coefficients of the energy functional $ \mathcal{D}_\eta (t) $.
\end{remark}

\begin{proof}

  For all $1 \leq k \leq s$, taking $\nabla^k$ on the third orientation equation of \eqref{PHLC}, multiplying $\nabla^k \dr$ by integration by parts, we have
  \begin{equation}\label{Global-1}
    \begin{aligned}
      & \tfrac{1}{2} \tfrac{\d}{\d t} \Big{(} - \lambda_1 |\nabla^k \dr|^2_{L^2} + \rho_1 |\nabla^k ( \dot{\dr} + \dr )|^2_{L^2} - \rho_1 |\nabla^k \dot{\dr}|^2_{L^2} - \rho_1 |\nabla^k \dr|^2_{L^2} \Big{)} + \tfrac{1}{2} |\nabla^{k+1} \dr|^2_{L^2} \\
      & - \rho_1 |\nabla^k \dot{\dr} + (\nabla^k \B) \dr + \tfrac{\lambda_2}{\lambda_1} (\nabla^k \A) \dr|^2_{L^2} +  \rho_1 | (\nabla^k \B) \dr + \tfrac{\lambda_2}{\lambda_1} (\nabla^k \A) \dr|^2_{L^2} \\
      =& - 2 \rho_1 \l \nabla^k \dot{\dr} , (\nabla^k \B) \dr + \tfrac{\lambda_2}{\lambda_1} (\nabla^k \A) \dr \r \\
        & - \rho_1 \sum_{\substack{a+b=k \\ a \geq 1}} \l \nabla^a \v \nabla^{b+1} \dot{\dr} , \nabla^k \dr \r - \rho_1 \sum_{\substack{a+b=k \\ a \geq 1}} \l \nabla^a \v \nabla^{b+1} \dr , \nabla^k \dot{\dr} \r \\
      & + \lambda_1 \l \nabla^k ( \v \cdot \nabla \dr ) , \nabla^k \dr \r  + \l \nabla^k (\gamma \dr), \nabla^k \dr \r + \lambda_1 \l \nabla^k (\B \dr) , \nabla^k \dr \r + \lambda_2 \l \nabla^k (\A \dr) , \nabla^k \dr \r \,.
    \end{aligned}
  \end{equation}
  Here we take advantage of the following relations
    \begin{align*}
      \rho_1 \l \nabla^k \ddot{\dr}, \nabla^k \dr \r = & \rho_1 \tfrac{\d}{\d t} \l \nabla^k \dot{\dr} , \nabla^k \dr \r - \rho_1 \l \nabla^k \dot{\dr}, \partial_t \nabla^k \dr \r + \rho_1 \l \nabla^k (\v \cdot \nabla \dot{\dr}), \nabla^k \dr \r \\
      =& \tfrac{1}{2} \tfrac{\d}{\d t}  \big{(}  \rho_1 | \nabla^k \dot{\dr} + \nabla^k \dr |^2_{L^2} -  \rho_1 |\nabla^k \dot{\dr}|^2_{L^2} -  \rho_1 |\nabla^k \dr|^2_{L^2} \big{)} - \rho_1 \l \nabla^k \dot{\dr}, \partial_t \nabla^k \dr \r \\
      & - \rho_1 \l \nabla^k \dot{\dr} , \v \cdot \nabla (\nabla^k \dr) \r + \rho_1 \sum_{\substack{a+b=k \\ a \geq 1}} \l \nabla^a \v \nabla^{b+1} \dot{\dr} , \nabla^k \dr \r \\
      =& \tfrac{1}{2} \tfrac{\d}{\d t}  \big{(}  \rho_1 | \nabla^k \dot{\dr} + \nabla^k \dr |^2_{L^2} -  \rho_1 |\nabla^k \dot{\dr}|^2_{L^2} -  \rho_1 |\nabla^k \dr|^2_{L^2} \big{)} - \rho_1 |\nabla^k \dot{\dr}|^2_{L^2} \\
      &  +  \rho_1 \sum_{\substack{a+b=k \\ a \geq 1}} \l \nabla^a \v \nabla^{b+1} \dot{\dr} , \nabla^k \dr \r + \rho_1 \sum_{\substack{a+b=k \\ a \geq 1}} \l \nabla^a \v \nabla^{b+1} \dr , \nabla^k \dot{\dr} \r \,,
    \end{align*}
  and
  \begin{equation*}
    \begin{aligned}
      - \rho_1 |\nabla^k \dot{\dr}|^2_{L^2} = & - \rho_1  |\nabla^k \dot{\dr} + (\nabla^k \B) \dr + \tfrac{\lambda_2}{\lambda_1} (\nabla^k \A) \dr |^2_{L^2} + \rho_1 | (\nabla^k \B) \dr + \tfrac{\lambda_2}{\lambda_1} (\nabla^k \A) \dr |^2_{L^2} \\
      & + 2 \rho_1 \l \nabla^k \dot{\dr}, (\nabla^k \B) \dr + \tfrac{\lambda_2}{\lambda_1} (\nabla^k \A) \dr \r \,.
    \end{aligned}
  \end{equation*}

  Now, for all $1 \leq k \leq s$, we estimate \eqref{Global-1} term by term.

  It is easy to be derived
    \begin{align}\label{Global-2}
      \no & - 2 \rho_1 \l \nabla^k \dot{\dr}, (\nabla^k \B) \dr + \tfrac{\lambda_2}{\lambda_1} (\nabla^k \A) \dr \r \leq 2 \rho_1 |\nabla^k \dot{\dr}|_{L^2} \cdot ( 1 - \tfrac{|\lambda_2|}{ \lambda_1 } ) |\nabla^{k+1} \v|_{L^2} \\
      \no \leq & 2 \rho_1 \Big{[} | \nabla^k \dot{\dr} + (\nabla^k \B) \dr + \tfrac{\lambda_2}{\lambda_1} (\nabla^k \A) \dr |_{L^2} + ( 1 - \tfrac{|\lambda_2|}{ \lambda_1 } ) |\nabla^{k+1} \v|_{L^2} \Big{]} \cdot ( 1 - \tfrac{|\lambda_2|}{ \lambda_1 } ) |\nabla^{k+1} \v|_{L^2} \\
      \leq& \rho_1 | \nabla^k \dot{\dr} + (\nabla^k \B) \dr + \tfrac{\lambda_2}{\lambda_1} (\nabla^k \A) \dr |^2_{L^2} + 3 \rho_1 ( 1 - \tfrac{|\lambda_2|}{ \lambda_1 } )^2 |\nabla^{k+1} \v|^2_{L^2}\,.
    \end{align}
  By H\"older inequality and Sobolev embedding, we yield
  \begin{equation}\label{Global-3}
    \begin{aligned}
      & - \rho_1 \sum_{\substack{a+b=k \\ a \geq 1}} \l \nabla^a \v \nabla^{b+1} \dot{\dr} , \nabla^k \dr \r  =  - \rho_1 \l \nabla \v \nabla^k \dot{\dr}, \nabla^k \dr \r  - \rho_1 \sum_{\substack{a+b=k \\ a \geq 2}} \l \nabla^a \v \nabla^{b+1} \dot{\dr} , \nabla^k \dr \r \\
      & \lesssim  \rho_1 |\nabla \v|_{L^\infty} |\nabla^k \dot{\dr}|_{L^2} |\nabla^k \dr|_{L^2} + \rho_1 \sum_{\substack{a+b=k \\ a \geq 2}} | \nabla^a \v|_{L^4} | \nabla^{b+1} \dot{\dr} |_{L^4} | \nabla^k \dr |_{L^2} \\
      & \lesssim \rho_1 |\nabla \v|_{H^s} |\dot{\dr}|_{H^s} |\nabla \dr|_{H^s} \,,
    \end{aligned}
  \end{equation}
  and similarly
  \begin{equation}\label{Global-3*}
    - \rho_1 \sum_{\substack{a+b=k \\ a \geq 1}} \l \nabla^a \v \nabla^{b+1} \dr , \nabla^k \dot{\dr} \r \lesssim \rho_1 |\nabla \v|_{H^s} |\dot{\dr}|_{H^s} |\nabla \dr|_{H^s} \,,
  \end{equation}
  and
  \begin{equation}\label{Global-4}
    \begin{aligned}
      & \lambda_1 \l \nabla^k ( \v \cdot \nabla \dr ) , \nabla^k \dr \r = \lambda_1 \sum_{\substack{a+b=k \\ a \geq 1}} \l \nabla^a \v \nabla^{b+1} \dr, \nabla^k \dr \r \\
      \lesssim & |\lambda_1| \sum_{\substack{a+b=k \\ a \geq 1}} |\nabla^a \v|_{L^2} |\nabla^{b+1} \dr|_{L^3} |\nabla^k \dr|_{L^6} \lesssim |\lambda_1| |\nabla \v|_{H^s} |\nabla \dr|_{\dot{H}^s} |\nabla \dr|_{H^s}\,.
    \end{aligned}
  \end{equation}
  Next we estimate the term $\l \nabla^k (\gamma \dr), \nabla^k \dr \r$. Since $\dr \cdot \dr = 1$, we have
  \begin{equation*}
    \begin{aligned}
      |\dr \cdot \nabla^k \dr|_{L^2} \lesssim \tfrac{1}{2} \sum_{\substack{p+q=k \\ p,q\geq 1}} |\nabla^p \dr \nabla^q \dr|_{L^2} \lesssim  \sum_{\substack{p+q=k \\ p,q\geq 1}} |\nabla^p \dr |_{L^3} | \nabla^q \dr|_{L^6} \lesssim |\nabla \dr|_{\dot{H}^s} |\nabla \dr|_{H^s}\,.
    \end{aligned}
  \end{equation*}
  Thus
    \begin{align*}
      & \l \nabla^k (\gamma \dr), \nabla^k \dr \r = \l \nabla^k \gamma, \dr \cdot \nabla^k \dr \r + \sum_{\substack{a+b=k \\ b \geq 1}} \l \nabla^a \gamma \nabla^b \dr, \nabla^k \dr \r \\
      \lesssim & |\nabla^k \gamma|_{L^2} |\dr \cdot \nabla^k \dr|_{L^2} + \sum_{\substack{a+b=k \\ b \geq 1}} | \nabla^a \gamma|_{L^2} | \nabla^b \dr|_{L^3} | \nabla^k \dr |_{L^6} \lesssim |\gamma|_{H^k} |\nabla \dr|_{\dot{H}^s} |\nabla \dr|_{H^s} \,.
    \end{align*}
  Via the estimation
  \begin{equation*}
    \begin{aligned}
      |\nabla^k \gamma|_{L^2} \lesssim & \rho_1 \sum_{a+b=k} |\nabla^a \dot{\dr}|_{L^3} |\nabla^b \dot{\dr}|_{L^6} + \sum_{a+b=k} |\nabla^{a+1} \dr|_{L^3} |\nabla^{b+1} \dr|_{L^6} \\
      & + |\lambda_2| \sum_{\substack{a+b+c=k \\ b,c \geq 1}} |\nabla^{a+1} \v \nabla^b \dr \nabla^c \dr|_{L^2} + |\lambda_2| \sum_{\substack{a+b=k \\ b \geq 1}} |\nabla^{a+1} \v \nabla^b \dr |_{L^2} + |\lambda_2| |\nabla^{k+1} \v|_{L^2} \\
      \lesssim & \rho_1 |\dot{\dr}|^2_{H^k} + |\nabla \dr|_{\dot{H}^k} |\nabla \dr|_{H^k} + |\lambda_2| |\nabla \v|_{H^s} + |\lambda_2| |\nabla \dr|_{\dot{H}^s} |\nabla \v|_{H^s} ( 1 + |\nabla \dr|_{H^s} ) \,,
    \end{aligned}
  \end{equation*}
  we know that
  \begin{equation}\label{Global-5}
    \begin{aligned}
       \l \nabla^k (\gamma \dr), \nabla^k \dr \r \lesssim ( \rho_1 |\dot{\dr}|^2_{H^s} + |\nabla \dr|^2_{\dot{H}^s} ) |\nabla \dr|^2_{H^s} + |\lambda_2| |\nabla \v|_{H^s} |\nabla \dr|_{\dot{H}^s} \sum_{p=1}^3 |\nabla \dr|^p_{H^s}\,.
    \end{aligned}
  \end{equation}

  By utilizing H\"older inequality, Sobolev embedding theory and the interpolation inequality $|f|_{L^4(\R^n)} \lesssim  |f|^{1-\frac{n}{4}}_{L^2(\R^n)} |\nabla f|^\frac{n}{4}_{L^2(\R^n)} $ for $n=2,3$, we have
    \begin{align}\label{Global-6}
      \no & \lambda_1 \l \nabla^k(\B \dr) , \nabla^k \dr \r =  - \lambda_1 \l (\nabla^{k} \B) \dr,  \nabla^{k+1} \dr \r -  \lambda_1 \sum_{\substack{a+b=k \\ b \geq 1}} \l \nabla^a \B \nabla^b \dr , \nabla^k \dr \r \\
      \lesssim & |\lambda_1| |\nabla^k \v|_{L^2} |\nabla^{k+1} \dr|_{L^2} + |\lambda_1| \sum_{\substack{a+b=k \\ b \geq 1}} |\nabla^{a+1} \v|_{L^2} |\nabla^b \dr|_{L^4} |\nabla^k \dr|_{L^4} \\
      \no \lesssim & |\lambda_1| |\nabla^k \v|_{L^2} |\nabla^{k+1} \dr|_{L^2}  + |\lambda_1| \sum_{\substack{a+b=k \\ b \geq 1}} |\nabla^{a+1} \v|_{L^2} |\nabla^b \dr|^{1 - \tfrac{n}{4}}_{L^2} | \nabla^{b+1} \dr |^\frac{n}{4}_{L^2} |\nabla^k \dr|^{1 - \tfrac{n}{4}}_{L^2} | \nabla^{k+1} \dr |^\frac{n}{4}_{L^2} \\
      \no \lesssim & |\lambda_1||\nabla \v|_{H^s} |\nabla \dr|_{\dot{H}^s} + |\lambda_1| |\nabla \v|_{H^s} |\nabla \dr|_{\dot{H}^s} |\nabla \dr|_{H^s}
    \end{align}
  for $n=2,3$, and similarly we have
  \begin{equation}\label{Global-7}
    \begin{aligned}
      \lambda_2 \l \nabla^k (\A \dr), \nabla^k \dr \r \lesssim |\lambda_2||\nabla \v|_{H^s} |\nabla \dr|_{\dot{H}^s} + |\lambda_2| |\nabla \v|_{H^s} |\nabla \dr|_{\dot{H}^s} |\nabla \dr|_{H^s} \,.
    \end{aligned}
  \end{equation}

  Noticing that
  \begin{equation*}
    C ( |\lambda_1| + |\lambda_2| ) |\nabla \v|_{H^s} |\nabla \dr|_{\dot{H}^s} \leq \tfrac{1}{4} |\nabla \dr|^2_{\dot{H}^s} + C^2 ( |\lambda_1| + |\lambda_2|  )^2 |\nabla \v|^2_{H^s} \,,
  \end{equation*}
  and by plugging the inequalities \eqref{Global-2}, \eqref{Global-3}, \eqref{Global-3*}, \eqref{Global-4}, \eqref{Global-5}, \eqref{Global-6}, and \eqref{Global-7} into \eqref{Global-1}, we obtain
  \begin{equation}\label{Global-New-Est}
    \begin{aligned}
      & \tfrac{1}{2} \tfrac{\d}{\d t} \Big{(} - \lambda_1 |\nabla \dr|^2_{H^{s-1}} - \rho_1 |\nabla \dr|^2_{H^{s-1}} + \rho_1 |\dot{\dr} + \dr|^2_{\dot{H}^s} - \rho_1 |\dot{\dr}|^2_{\dot{H}^s} \Big{)} + \tfrac{1}{4} |\nabla \dr|^2_{\dot{H^s}} \\
      & - 2 \rho_1 \sum_{k=1}^s |\nabla^k \dot{\dr} + (\nabla^k \B) \dr + \tfrac{\lambda_2}{\lambda_1} (\nabla^k \A) \dr|^2_{L^2} +  \rho_1 \sum_{k=1}^s | (\nabla^k \B) \dr + \tfrac{\lambda_2}{\lambda_1} (\nabla^k \A) \dr|^2_{L^2} \\
      & - \big{[} 3 \rho_1 ( 1 - \tfrac{|\lambda_2|}{\lambda_1} )^2 + C^2 ( |\lambda_1| + |\lambda_2| )^2 \big{]} |\nabla \v|^2_{H^s} \\
      \lesssim & ( \rho_1 |\dot{\dr}|^2_{H^s} + |\nabla \dr|^2_{\dot{H}^s} ) |\nabla \v|_{H^s} + \rho_1 |\nabla \v|_{H^s} |\dot{\dr}|_{H^s} |\nabla \dr|_{H^s} \\
      & + ( |\lambda_1| + |\lambda_2|  ) |\nabla \v|_{H^s} |\nabla \dr|_{\dot{H}^s} \sum_{p=1}^3  |\nabla \dr|^p_{H^s} \,.
    \end{aligned}
  \end{equation}

  Choosing
  \begin{equation*}
    \eta_0 = \tfrac{1}{2} \min \Big{\{} \tfrac{ \frac{1}{2} \mu_4 }{ 3 \rho_1 ( 1 - \frac{|\lambda_2|}{\lambda_1} )^2 + C^2 ( |\lambda_1| + |\lambda_2| )^2  } \,, \tfrac{- \lambda_1}{ 2 \rho_1 } \,, \tfrac{1}{\rho_1} \,, 1 \Big{\}} \in (0, \tfrac{1}{2} ] \,,
  \end{equation*}
  and multiplying the inequality \eqref{Global-New-Est} by $\eta \in ( 0, \eta_0 ]$ and then adding it to the inequality \eqref{APE-Hs}, we derive
    \begin{align}\label{Global}
      \no & \tfrac{1}{2} \tfrac{\d}{\d t} \Big{(} |\v|^2_{H^s} + ( - \eta \lambda_1 + 1 - \eta \rho_1 ) |\nabla \dr|^2_{H^{s-1}} + |\nabla^{s+1} \dr|^2_{L^2} \\
      \no & \qquad + \rho_1 (1-\eta) |\dot{\dr}|^2_{\dot{H}^s} + \rho_1 |\dot{\dr}|^2_{L^2} + \eta \rho_1 |\dot{\dr} + \dr|^2_{\dot{H}^s} \Big{)} \\
      \no & + \Big{[} \tfrac{1}{2} \mu_4 - \eta \big{[} 3 \rho_1 ( 1 -\tfrac{|\lambda_2|}{\lambda_1} )^2 + C^2 ( |\lambda_1| + |\lambda_2| )^2 \big{]} \Big{]} |\nabla \v|^2_{H^s} + \tfrac{1}{4} \eta |\nabla \dr|^2_{\dot{H}^s} \\
      \no & + ( - \lambda_1 - 2 \eta \rho_1 ) \sum_{k=0}^s |\nabla^k \dot{\dr} + (\nabla^k \B) \dr + \tfrac{\lambda_2}{\lambda_1} (\nabla^k \A) \dr|^2_{L^2} + \mu_1 \sum_{k=0}^s |\dr^\top (\nabla^k \A) \dr|^2_{L^2} \\
       & + ( \mu_5 + \mu_6 + \tfrac{\lambda_2^2}{\lambda_1}  ) \sum_{k=0}^s |(\nabla^k \A) \dr|^2_{L^2} + \eta \rho_1 \sum_{k=1}^s | (\nabla^k \B) \dr + \tfrac{\lambda_2}{\lambda_1} (\nabla^k \A) \dr|^2_{L^2} \\
      \no \lesssim & \big{(} |\v|^2_{\dot{H}^s} + \rho_1 |\dot{\dr}|^2_{H^s} + |\nabla \dr|^2_{\dot{H}^s} \big{)} |\nabla \v|_{H^s} + \rho_1 |\nabla \v|_{H^s} |\dot{\dr}|_{H^s} |\nabla \dr|_{H^s} \\
      \no & + \rho_1 |\dot{\dr}|^3_{H^s} |\nabla \dr|_{H^s} + |\dot{\dr}|_{H^s} |\nabla \dr|_{\dot{H}^s} |\nabla \dr|^2_{H^s} \\
      \no & + ( \mu_1 + |\mu_2| + |\mu_3| + |\mu_5| + |\mu_6| ) \sum_{p=1}^4 |\nabla \dr|^p_{H^s} (  |\v|_{\dot{H}^s} + |\dot{\dr}|_{H^s} + |\nabla \dr|_{\dot{H}^s} ) |\nabla \v|_{H^s}\,.
    \end{align}

  Noticing that
  \begin{equation*}
    |\nabla^k \dot{\dr}|_{L^2} \leq |\nabla^k \dot{\dr} + (\nabla^k \B) \dr + \tfrac{\lambda_2}{\lambda_1} (\nabla^k \A) \dr|_{L^2} + ( 1 - \tfrac{|\lambda_2|}{\lambda_1} ) |\nabla^{k+1} \v|_{L^2} \,,
  \end{equation*}
  then the estimate \eqref{Global} reduce to the conclusion of Lemma \ref{Lm-Global}.

\end{proof}

\noindent{\bf Proof of the second part of Theorem \ref{Main-Thm}: Global existence with small initial data.} We observe that
\begin{equation*}
  C_\# \big{(} |\v|^2_{H^s} + \rho_1 |\dot{\dr}|^2_{H^s} + |\nabla \dr|^2_{H^s} \big{)} \leq \mathcal{E}_\eta (t) \leq C^\# \big{(} |\v|^2_{H^s} + \rho_1 |\dot{\dr}|^2_{H^s} + |\nabla \dr|^2_{H^s} \big{)} \,,
\end{equation*}
and
$$ \mathcal{D}_\eta (t) \geq \tfrac{1}{4} \mu_4 |\nabla \v|^2_{H^s} \,, $$
where the constants $ C^\# = 4 + 2 \eta_0 - \lambda_1 \eta_0 + 2 \rho_1 \eta_0 > 0 $ and $ C_\# = \min \{ 1 , 1 - \eta_0, 1 - \eta_0 \rho_1 \} > 0 $, since $\eta_0 > 0$ is sufficiently small. As a consequence,
$$ C_\# E^{in} \leq \mathcal{E}_\eta (0) \leq C^\# E^{in} \,. $$

We define the following number
$$ T^* = \sup \bigg{\{}  \tau >0; \sup\limits_{t \in [0,\tau]} C \sum\limits_{p=1}^4 \mathcal{E}_\eta^\frac{p}{2} (t) \leq \tfrac{1}{2} \bigg{\}} \geq 0 \,, $$
where the constant $C > 0$ is mentioned as in Lemma \ref{Lm-Global}. We choose the positive number $\eps_0 \equiv \frac{1}{C^\#} \min \Big{\{} 1 , \frac{1}{256 C^2 } \Big{\}} > 0$. If the initial energy $E^{in} \leq \eps_0$, we can deduce that
$$ C \sum\limits_{p=1}^4 \mathcal{E}_\eta^\frac{p}{2} (0) \leq \tfrac{1}{4} < \tfrac{1}{2} \,,  $$
which derives that $ T^* > 0 $ from the continuity of the energy functional $\mathcal{E}_\eta (t)$. Thus for all $t \in [0,T^*]$
\begin{equation*}
  \tfrac{\d}{\d t} \mathcal{E}_\eta (t) + \bigg{[} 1 - C \sum\limits_{p=1}^4 \mathcal{E}_\eta^\frac{p}{2} (t) \bigg{]} \mathcal{D}_\eta (t) \leq 0 \,,
\end{equation*}
which immediately means that $\mathcal{E}_\eta (t) \leq \mathcal{E}_\eta (0) \leq C^\# E^{in}$ holds for all $t \in [0,T^*]$, and consequently
$$  \sup_{t \in [0,T^*]} \bigg{\{} C \sum\limits_{p=1}^4 \mathcal{E}_\eta^\frac{p}{2} (t) \bigg{\}} \leq \tfrac{1}{4}\,.  $$
So we claim that $T^* = + \infty$. Otherwise, if $T^* < + \infty$, the continuity of the energy $\mathcal{E}_\eta (t)$ implies that there is a small positive $\eps > 0$ such that
$$ \sup_{t \in [0,T^* + \eps]} \bigg{\{} C \sum\limits_{p=1}^4 \mathcal{E}_\eta^\frac{p}{2} (t) \bigg{\}} \leq \tfrac{3}{8} < \tfrac{1}{2} \,, $$
which contradicts to the definition of $T^*$. Therefore we gain
$$  \sup_{t \geq 0} \Big{(} |\v|^2_{H^s} + \rho_1 |\dot{\dr}|^2_{H^s} + |\nabla \dr|^2_{H^s} \Big{)} (t) + \tfrac{1}{2} \mu_4 \int_0^\infty |\nabla \v|^2_{H^s} \d t \leq C_3 E^{in} \,,$$
where $C_3 >0$ depends only on the Leslie coefficients and inertia density constant $\rho_1$, and as a consequence, the proof of the second part of Theorem \ref{Main-Thm} is finished.

\section{Conclusions}

In this paper, we prove local well-posedness of the small solutions to the Ericksen-Leslie's hyperbolic liquid crystal model \eqref{PHLC} under the assumptions on the Leslie coefficients which ensure the dissipation of the basic energy law. Then with an additional assumption $\lambda_1 < 0$, we prove the global existence of the smooth solutions. In fact, this paper is only the starting point of the research of the general system \eqref{PHLC} in high dimension. Many problems are left to be investigated in future. For examples:
\begin{enumerate}
  \item The assumption $\lambda_1< 0$ excludes the most important special case \eqref{NS-WM}, which is the coupled incompressible Navier-Stokes equations with the wave map to unit sphere. For that case, $\lambda_1 = 0$. We expect that the global smooth solutions with small initial data should exist.

  \item Physically, the inertial constant $\rho_1$ is small. Formally, letting $\rho_1=0$ will gives the parabolic Ericksen-Leslie's system which has been studied intensively in the past 30 years. Justification of the limit $\rho_1\rightarrow 0$ from hyperbolic to parabolic Ericksen-Leslie's system will be an interesting question. In a paper under preparation joint with S-J. Tang and A. Zarnescu \cite{Jiang-Luo-Tang-Zarnescu}, we can justify this limit for $\v(x,t)\equiv 0$, i.e. from wave map type to the heat flow type equation in the context of smooth solutions local in time. The initial layer will arise in this limit.

  \item The existence of the global weak solution of \eqref{PHLC} is another interesting question. The geometric constraint $|\dr|=1$ will bring serious analytical difficulty.
\end{enumerate}


\section*{Acknowledgement}
We first appreciate Prof. Fanghua Lin, who suggested the Ericksen-Leslie's hyperbolic liquid crystal model to us when we visited  the NYU-ECNU Institute of Mathematical Sciences at NYU Shanghai during the spring semester of 2015. We also thanks the hospitality of host institute. We also thanks for the conversations with Zhifei Zhang and Wei Wang, in particular on their work \cite{Wang-Zhang-Zhang-ARMA2013}, which shares light on the current work. During the preparation of this paper, Xu Zhang made several valuable comments and suggestions, we take this opportunity to thank him here.

\bigskip



\end{document}